\documentclass{article}
\usepackage{amsfonts,amsmath,amssymb,amsthm,amstext,txfonts}
\usepackage{extarrows,longtable}
\usepackage{graphicx,hyperref,mathrsfs,url}
\usepackage{geometry}
\usepackage{titlesec}
\usepackage[title]{appendix}

\allowdisplaybreaks

\begin{document}
\title{Weak parabolic Harnack inequality and H\"older regularity for non-local Dirichlet forms}
\author{Guanhua Liu\footnote{Fakult\"at f\"ur Mathematik, Universit\"at Bielefeld, Postfach 100131, D-33501 Bielefeld, Germany. E-mail: gliu@math.uni-bielefeld.de\\The research is supported by German Research Foundation (DFG) SFB 1283}}
\date{February 6, 2025}
\maketitle

\newtheorem{theorem}{Theorem}[section]
\newtheorem{proposition}[theorem]{Proposition}
\newtheorem{lemma}[theorem]{Lemma}
\newtheorem{remark}[theorem]{Remark}
\newtheorem{example}[theorem]{Example}
\newtheorem{corollary}[theorem]{Corollary}
\numberwithin{equation}{section}

\begin{abstract}
In this paper we give equivalent conditions for the weak parabolic Harnack inequality for general regular Dirichlet forms without killing part, in terms of local heat kernel estimates or growth lemmas. With a tail estimate on the jump measure, we obtain from these conditions the H\"older continuity of caloric and harmonic functions. Our results generalize the theory of Chen, Kumagai and Wang, in the sense that the upper jumping smoothness condition (UJS) is canceled. We also derive the complete forms of Harnack inequalities from the globally non-negative versions, and obtain continuity of caloric functions with worse tails.
\end{abstract}

\section{Introduction}

Harnack inequalities are widely discussed in the theory of partial differential equations and Markov processes, controlling the ratio between the supremum and infimum of a harmonic or caloric function in a regular domain by a constant independent of the function. The elliptic Harnack inequality (EHI) was first raised by C. G. von Harnack in 1887 for the real part of a complex analytic function, while the parabolic Harnack inequality (PHI) by B. Pini \cite{pi} for classical solutions of the heat equation. The versions for weak solutions were established by E. de Giorgi \cite{dg} (elliptic) and J. Moser \cite{mos-p} (parabolic), from which the inner regularity follows, see J. Nash \cite{nas} and Moser \cite{mos-e,mos+}. The related results constitute the so-called DeGiorgi-Nash-Moser theory of partial differential equations. We can refer to \cite{ka} for more about the early history of Harnack inequalities.

EHI, PHI and related properties can also be defined for various operators on different types of spaces. We particularly mention that there are Harnack inequalities for non-local operators, which have been widely discussed in \cite{ck,ckw-e,ckw-p,dckp,kw} for example. Here EHI and PHI make sense for globally non-negative solutions only, or alternatively, for those solutions that are only non-negative on the target region, a tail of the negative part must be added to the infimum side.

Obviously PHI implies EHI, while several different examples were given in \cite{del+} where EHI holds but PHI fails. Till now there is no non-trivial equivalent characterization of EHI, but it is known to be stable under quasi-isometries according to \cite{bm} based on some standard assumptions (especially strong locality).

On the other hand, there is a very convenient equivalent expression of PHI in terms of heat kernel estimates in many cases (local or non-local), for example, diffusions on Riemannian manifolds \cite{gri,sc} and $\alpha$-stable jump processes on metric spaces \cite{ck}. A complete theory on this topic for general local Dirichlet spaces is given in \cite{bbk} and \cite{ghl}, respectively in the language of probability and analysis. A non-local counterpart is given by \cite{ckw-e,ckw-p}, only in the language of probability. This paper serves as a first (and major) part of the analytic analog, mainly concerning the weak parabolic Harnack inequality.

Following Moser's idea in \cite{mos-e,mos-p}, EHI and PHI are usually proved by a mean-value inequality (EMV for elliptic case, PMV for parabolic), controlling the upper bound by the average, and a weak Harnack inequality (wEH, wPH respectively), controlling the average by the lower bound. Theory on mean-value inequalities is established for general Dirichlet forms in \cite{ghl} (local, elliptic), \cite{ab} (local, parabolic), \cite{ghh+1} (non-local, elliptic) and \cite{ghh+2} (non-local, parabolic). On the other hand, wEH and wPH are usually proved by EMV and PMV applied on a transform of the (super)solution (say in \cite{kk}), or by estimating measures of level sets. The latter approach focuses on the ``growth lemma'', which was raised by Landis \cite{land} for elliptic equations, and then by Krylov and Safonov \cite{ks} for parabolic ones, particularly for non-divergence differential operators.

On general Dirichlet spaces (especially non-local ones), wEH is proved in \cite{hy} under very strong conditions, and the only equivalent characterization given there is an elliptic growth lemma. Meanwhile, for wPH there has been neither sufficient nor necessary condition raised -- it has even never appeared in the general setting. On $\mathbb{R}^d$, a proof of wPH is given by \cite[Theorem 1.1]{fk} for $\alpha$-stable non-local operators, and also in \cite{gk,is,str} under similar settings. However, in these papers the jump measure is assumed to be bounded from both sides, which is too strong for wPH, or even for PHI: firstly, such assumptions can not cover strongly local forms and $\rho$-local forms (i.e., the jump measure vanishes between any two sets with distance greater than $\rho$); secondly, PHI holds for some subordinated diffusions whose jump measure has a different scale (see \cite[Corollary 2.5]{lm2}).

In this paper, we give general equivalent criteria for wPH in Theorems \ref{M1} and \ref{M2}. They are essential in the sense that PHI may fail under these conditions. To be exact, we remove the upper jumping smoothness condition (UJS), which is necessary for PHI by \cite[Proposition 3.3]{ckw-p}. We will prove the two theorems respectively in Sections \ref{PM1} and \ref{PM2}. Further properties related to wPH are discussed in Sections \ref{Tail} and \ref{ctn}, regulating the tail term and yielding a more general continuity criterion respectively. These properties help connect our results with the stochastic theory \cite{ckw-e,ckw-p} on PHI. In particular, we obtain in Corollary \ref{M4} the relation between weak and strong parabolic Harnack inequalities.

\section{Notions and main results}
Throughout this paper, we assume that $(M,d)$ is a locally compact, complete and separable metric space, and $\mu$ is a Radon measure of full support on $M$. Denote
$$\overline{R}:=2\ \mathrm{diam}(M,d)\in[0,\infty].$$
For any $x\in M$ and $r>0$, let the metric ball $B(x,r)$ be the set $\{y\in M:d(x,y)<r\}$ equipped with fixed center $x$ and radius $r$. Note that every metric ball is precompact under our assumptions. For any $0<\lambda<\infty$ and any ball $B=B(x,r)$, let
$$\lambda B:=B(x,\lambda r)\quad\mbox{and}\quad V(x,r):=\mu(B(x,r)).$$

The following two conditions are always assumed in the sequel:

\begin{itemize}
\item $(\mathrm{VD})$, the \emph{volume doubling} property: there exists $C\ge 1$ such that $V(x,2R)\le CV(x,R)$ for all $x\in M$ and $R>0$. Equivalently, there exist $\alpha,C>0$ such that for all $x,y\in M$ and $0<r\le R<\infty$,
$$\frac{V(x,R)}{V(y,r)}\le C\left(\frac{R+d(x,y)}{r}\right)^\alpha.$$

\item $(\mathrm{RVD})$, the \emph{reverse volume doubling} property: there exists $C\ge 1$ such that $V(x,R)\ge 2V(x,C^{-1}R)$ for all $x\in M$ and $0<R<\overline{R}$. Equivalently, there exist $\alpha',C>0$ such that for all $x\in M$ and $0<r\le R<\overline{R}$,
$$\frac{V(x,R)}{V(x,r)}\ge C^{-1}\left(\frac{R}{r}\right)^{\alpha'}.$$
\end{itemize}

Let $\mathcal{F}$ be a dense subspace of $L^2(M,\mu)$, and $\mathcal{E}$ be a symmetric non-negative semi-definite quadratic form on $\mathcal{F}$. We call $(\mathcal{E},\mathcal{F})$ a \emph{Dirichlet form} on $L^2(M,\mu)$, if
\begin{itemize}
\item $\mathcal{F}$ is a Hilbert space equipped with the inner product $\mathcal{E}_1(u,v):=\mathcal{E}(u,v)+(u,v)_{L^2(\mu)}$;
\item the Markovian property holds: $f_+\wedge 1\in\mathcal{F}$ for every $f\in\mathcal{F}$, and $\mathcal{E}(f_+\wedge 1,f_+\wedge 1)\le\mathcal{E}(f,f)$.
\end{itemize}
It is called \emph{regular}, if $\mathcal{F}\cap C_0(M,d)$ is dense both in $\mathcal{F}$ (under norm $\mathcal{E}_1^{1/2}$) and in $C_0(M,d)$ (under norm $\|\cdot\|_\infty$). In this case, for any open set $\Omega\subset M$ and any Borel subset $U\Subset\Omega$, there exists $\phi\in\mathcal{F}\cap C_0(\Omega)$ such that $\phi=1$ on $U$ and $0\le\phi\le 1$ throughout $M$. The collection of such $\phi$ is denoted by $\mathrm{cutoff}(U,\Omega)$.

Every regular Dirichlet form admits a Beurling-Deny decomposition as follows:
\begin{equation}\label{bd}
\mathcal{E}(u,v)=\mathcal{E}^{(L)}(u,v)+\int_{M\times M}(u(x)-u(y))(v(x)-v(y))dj(x,y)+\int_Muvd\kappa
\end{equation}
for all $u,v\in\mathcal{F}\cap C_0(\Omega)$, where $\mathcal{E}^{(L)}$ is \emph{strongly local}, that is, $\mathcal{E}^{(L)}(u,v)=0$ whenever there exists $a\in\mathbb{R}$ such that $\mathrm{supp}(u-a)\cap\mathrm{supp}(v)=\emptyset$; the \emph{jump measure} $j$ is a symmetric Radon measure on $M\times M$ with $j(\mathrm{diag})=0$; the \emph{killing measure} $\kappa$ is a Radon measure on $M$. See \cite[Theorem 4.5.2]{fot} for the meaning of these notions.

Throughout this paper we assume that $\kappa\equiv 0$, and there exists a Borel transition kernel $J(x,dy)$ such that
$$j(dx,dy)=d\mu(x)J(x,dy)=d\mu(y)J(y,dx).$$
In this case, $\mathcal{E}$ can be extended to $\mathcal{F}'$, the collection of all functions $u=v+a$, where $v\in\mathcal{F}$ and $a\in\mathbb{R}$, by $\mathcal{E}(u,u):=\mathcal{E}(v,v)$. Further, (\ref{bd}) holds for all $u,v\in\mathcal{F}'$.

Recall that for every $u\in\mathcal{F}$, there is a unique Radon measure $\Gamma(u)$ on $M$, called the (le Jan) \emph{energy form} of $u$, determined by
$$\int_M\varphi d\Gamma(u)=\mathcal{E}(u,\varphi u)-\frac{1}{2}\mathcal{E}(u^2,\varphi)\quad\mbox{for all}\quad\varphi\in\mathcal{F}.$$
Under our assumptions, $\Gamma(u)$ can also be defined for any $u\in\mathcal{F}'$ (with $\Gamma(c)\equiv 0$ for any constant $c$), and is decomposed into
$$d\Gamma(u)(x)=d\Gamma^{(L)}(u)(x)+\left(\int_M(u(x)-u(y))^2J(x,dy)\right)d\mu(x),$$
where $\Gamma^{(L)}$ is the energy form for $\mathcal{E}^{(L)}$.

For every open set $\Omega\subset M$, we also define the $\Omega$-local energy
$$d\Gamma_\Omega(u)(x)=1_\Omega(x)\left\{d\Gamma^{(L)}(u)(x)+\left(\int_\Omega(u(x)-u(y))^2J(x,dy)\right)d\mu(x)\right\}.$$
Let $\mathcal{F}(\Omega)$ be the closure of $\mathcal{F}\cap C_0(\Omega)$ in $(\mathcal{F},\mathcal{E}_1)$, and $P_t^\Omega$ be the \emph{heat semigroup} on $L^2(\Omega)$ associated with the Dirichlet form $(\mathcal{E},\mathcal{F})$, that is, for any $f\in L^2(\Omega)$ and $\varphi\in\mathcal{F}(\Omega)$,
$$\lim\limits_{t\downarrow 0}\left\|P_t^\Omega f-f\right\|_{L^2(\Omega)}=0\quad\mbox{and}\quad\frac{\mathrm{d}}{\mathrm{d}t}\left(P_t^\Omega f,\varphi\right)+\mathcal{E}\left(P_t^\Omega f,\varphi\right)=0\quad\mbox{for all}\quad t>0.$$
It can always be extended to a contractive semigroup on $L^p(\Omega)$ with $1\le p\le\infty$, still denoted as $P_t^\Omega$, and thus defining a Borel transition kernel $p_t^\Omega(x,dy)$ on $\Omega$ by
$$P_t^\Omega f(x)=\int_\Omega f(y)p_t^\Omega(x,dy)$$
for $\mu$-a.e.\ $x\in\Omega$ and all $f\in L^p(\Omega)$ with any $1\le p\le\infty$. Be careful that $\Gamma_\Omega$ is the energy form of the Neumann boundary problem on $\Omega$, but $P_t^\Omega$ is the heat semigroup of the Dirichlet boundary problem on $\Omega$. Generally they do not correspond to each other.

If $p_t^\Omega(x,dy)\ll d\mu(y)$, then we call the Radon-Nykodym derivative $p_t^\Omega(x,y)$ as the \emph{heat kernel} of $P_t^\Omega$. When $\Omega=M$, we simply write $P_t$ and $p_t(x,y)$ (if exists) instead of $P_t^M$ and $p_t^M(x,y)$, respectively called the (global) heat semigroup and heat kernel associated with the Dirichlet form $(\mathcal{E},\mathcal{F})$. We also define the Green operator
$$G^\Omega f:=\int_0^\infty P_t^\Omega fdt\quad\mbox{for all}\quad f\in L^2(\Omega).$$

Further, we define $\mathcal{F}'(\Omega)$ as the collection of all $v\in L_{loc}^1(M)$ such that there exist $w\in\mathcal{F}$ and a strict neighbourhood $\Omega'$ of $\Omega$ (that is, $K\cap\Omega\Subset\Omega'$ for any compact set $K$ in $M$) such that $v=w$ on $\Omega'$ and
$$\int_\Omega\int_M(v(x)-v(y))^2J(x,dy)d\mu(x)<\infty.$$
For all $v\in\mathcal{F}'(\Omega)$ and $\varphi\in\mathcal{F}(\Omega)$, define the extended Dirichlet form
$$\tilde{\mathcal{E}}(v,\varphi):=\mathcal{E}^{(L)}(w,\varphi)+\int_{M\times M}(\varphi(x)-\varphi(y))(v(x)-v(y))dj(x,y),$$
which is well-defined (that is, independent of the selection of $w$ and $\Omega'$) due to the strong locality of $\mathcal{E}^{(L)}$. The following facts are obvious:

\begin{itemize}
\item[i)]$\mathcal{F}(\Omega)\subset\mathcal{F}\subset\mathcal{F}'\subset\mathcal{F}'(\Omega)$, and $\tilde{\mathcal{E}}(v,\varphi)=\mathcal{E}(v,\varphi)$ whenever $v\in\mathcal{F}'$;

\item[ii)]$fg,f\wedge a,f+a\in\mathcal{F}'(\Omega)$ for any $f\in\mathcal{F}'(\Omega)$, $g\in\mathcal{F}'\cap L^\infty$ and $a\in\mathbb{R}$;

\item[iii)]$\tilde{\mathcal{E}}$ is $\mathbb{R}$-bilinear on ${F}'(\Omega)\times\mathcal{F}(\Omega)$.
\end{itemize}

Let $Q=(t_1,t_2]\times\Omega$ be an arbitrary cylinder in $\mathbb{R}\times M$, where $t_1<t_2$, and $\Omega$ is an open subset of $M$. For any function $u:(t_1,t_2]\to\mathcal{F}'(\Omega)$, let
$$\mathop{\mathrm{esup}}_Qu:=\sup\limits_{t_1<t\le t_2}\mathop{\mathrm{esup}}_\Omega u(t,\cdot),\quad\mathop{\mathrm{einf}}_Qu:=\inf\limits_{t_1<t\le t_2}\mathop{\mathrm{einf}}_\Omega u(t,\cdot)\quad\mbox{and}\quad\|u\|_Q=\sup\limits_{t_1<t\le t_2}\|u(t,\cdot)\|_{L^\infty(\Omega)}.$$

On $\mathbb{R}\times M$, let $\hat{\mu}$ be the completion of the product measure $m_1\otimes\mu$, where $m_1$ is the Lebesgue measure on $\mathbb{R}$. By \cite[Theorem 7.6.5]{bo}, $\hat{\mu}$ is a Radon measure with respect to the metric $\hat{d}\big((t,x),(t',x')\big):=|t-t'|\vee d(x,x').$

Given two open sets $U,\Omega$ in $M$ such that $U\Subset\Omega$, define the tail of a function $f$ with respect to $(U,\Omega)$ as
$$T_\Omega^U(f)=\mathop{\mathrm{esup}}\limits_{z\in U}\int_{\Omega^c}|f(y)|J(z,dy).$$
Set $Q=(t_1,t_2]\times\Omega$ with some $t_1<t_2$. For any $\hat{\mu}$-measurable function $u$ on $(t_1,t_2]\times M$, we write $T_\Omega^U(u(s,\cdot))$ simply as $T_\Omega^U(u;s)$, and define $T_Q^{U;p}(u)$, the $L^p$-tail of $u$ with respect to $(U,Q)$, as the $L^p$-norm of $T_\Omega^U(u;\cdot)$ on $(t_1,t_2]$, that is,
$$T_Q^{U;p}(u)=\left(\int_{t_1}^{t_2}\left(T_\Omega^U(u;s)\right)^pds\right)^{1/p}\quad(1\le p<\infty)\quad\quad\mbox{and}\quad\quad T_Q^{U;\infty}(u)=\mathop{\mathrm{esup}}\limits_{t_1<s\le t_2}T_\Omega^U(u;s).$$
For simplicity, we denote $T_Q^{U;1}(u)$ as $\overline{T}_Q^U(u)$. Further, if $\Omega=B$ and $U=\lambda B$ with some metric ball $B$ and $0<\lambda<1$, then we also write these tails as $T_B^{(\lambda)}(f)$, $T_Q^{(\lambda;p)}(u)$ and $\overline{T}_Q^{(\lambda)}(u)$ so that ``$B$'' appears only once.

We say $u:(t_1,t_2]\to\mathcal{F}'(\Omega)$ is \emph{caloric} (or \emph{subcaloric}, or \emph{supercaloric}) on $(t_1,t_2]\times\Omega$, if
\begin{itemize}
\item $u$ is $\hat{\mu}$-measurable on $(t_1,t_2]\times M$;
\item for any $\varphi\in L^2(\Omega)$, $I_\varphi(t):=(u(t,\cdot),\varphi)_{L^2(\Omega)}$ is continuous on $(t_1,t_2]$, and the left derivative $\frac{\mathrm{d}_-}{\mathrm{d}t}I_\varphi(t)$ exists at every $t_1<t\le t_2$. Further, if $\varphi\in\mathcal{F}(\Omega)$, then
\begin{equation}\label{cal}
\frac{\mathrm{d}_-}{\mathrm{d}t}\big(u(t,\cdot),\varphi\big)+\tilde{\mathcal{E}}(u(t,\cdot),\varphi)\quad=\quad\mbox{(or }\le\mbox{, or }\ge\mbox{)}\quad 0;
\end{equation}
\item there exist $t_1<s_0<\cdots<s_N=t_2$ (where $N<\infty$) such that the derivative $\frac{\mathrm{d}}{\mathrm{d}t}I_\varphi(t)$ exists for all $\varphi\in L^2(\Omega)$ at all $t\in(t_1,t_2]\setminus\{s_1,\dotsc,s_N\}$.
\end{itemize}

Note that here $t\mapsto(u(t,\cdot),\varphi)$ is piecewise differentiable, so that
\begin{equation}\label{cal'}
\big(u(t,\cdot),\varphi\big)-\lim_{s\downarrow t_1}\big(u(s,\cdot),\varphi\big)+\int_{t_1}^t\tilde{\mathcal{E}}(u(s,\cdot),\varphi)ds\quad=\quad\mbox{(or }\le\mbox{, or }\ge\mbox{)}\quad 0
\end{equation}
for all $\varphi\in\mathcal{F}(\Omega)$ and $t_1<t\le t_2$.

A function $f\in\mathcal{F}'(\Omega)$ is called \emph{(sub/super)harmonic} on $\Omega$, if $(t,x)\mapsto f(x)$ is (sub/super)caloric.

Note also if $u$ is (sub/super)caloric on $(t_1,t_2]\times\Omega$, then by the uniform boundedness theorem on $L^2(\Omega)$, there exists $v(t,\cdot)\in L^2(\Omega)$ for every $t_1<t\le t_2$ such that for all $\varphi\in L^2(\Omega)$,
$$\frac{\mathrm{d}_-}{\mathrm{d}t}(u(t,\cdot),\varphi)_{L^2(\Omega)}=(v(t,\cdot),\varphi)_{L^2(\Omega)}.$$
This way, we say $\partial_tu=v$ weakly in $\Omega$, and (\ref{cal}) is equivalently written as
$$(\partial_tu,\varphi)+\mathcal{E}(u,\varphi)\quad=\quad\mbox{(or }\le\mbox{, or }\ge\mbox{)}\quad 0.$$
On the other hand, if $t\mapsto(u(t,\cdot),\varphi)_{L^2(M)}$ is (piecewise) differentiable for all $\varphi\in L^2(M)$, then $u:(t_1,t_2]\to\mathcal{F}$ is always $\hat{\mu}$-measurable by Lemma \ref{meas}. Therefore, our definition contains \cite[Definition 2.8]{ghh+2} completely, and is more general than that. Additionally, in \cite{ckw-p} caloricity is defined in a stochastic way (we call $X$-caloric in Appendix \ref{prob}), which is greatly different with our analytic definition. A basic relation between (analytically) caloric and $X$-caloric functions is given in Proposition \ref{AIS}, but this theory is far from completed.

Let $W:M\times\mathbb{R}_+\to\mathbb{R}_+$ be a \emph{scaling function} (for short, a \emph{scale}), that is, given any $x\in M$, $W(x,\cdot)$ is strictly increasing, continuous, and there exist $C_W\ge 1$ and $0<\beta_1\le\beta_2<\infty$ such that for all $0<r\le R<\infty$ and all $x,y\in M$ with $d(x,y)\le R$,
\begin{equation}\label{W}
C_W^{-1}\left(\frac{R}{r}\right)^{\beta_1}\le\frac{W(x,R)}{W(y,r)}\le C_W\left(\frac{R}{r}\right)^{\beta_2}.
\end{equation}
It follows clearly that $W(x,0)=0$ and $W(x,R)\to\infty$ as $R\to\infty$ for any $x\in M$.

For a ball $B=B(x,r)$, we also denote $W(B):=W(x,r)$. For any $t\ge 0$, let $W^{-1}(x,t)$ be the unique $r\in\mathbb{R}_+$ such that $t=W(x,r)$. It follows by (\ref{W}) that for all $0<s\le t<\infty$ and $x\in M$,
\begin{equation}\label{W-1}
\left(\frac{t}{C_Ws}\right)^{\frac{1}{\beta_2}}\le\frac{W^{-1}(x,t)}{W^{-1}(x,s)}\le\left(\frac{C_Wt}{s}\right)^{\frac{1}{\beta_1}}.
\end{equation}

Fixing a scale $W$, we recall the following conditions on $J(x,dy)$:

\begin{itemize}
\item $(\mathrm{TJ})$, the tail estimate of jump measure: for every ball $B=B(x_0,R)$, $W(B)J(x_0,B^c)$ is uniformly bounded. Equivalently (since $B(x,(1-\lambda)R)\subset B$ whenever $x\in\lambda B$), for any $0<\lambda<1$, there exists a constant $C_\lambda>0$ such that for every ball $B=B(x_0,R)$,
$$\mathop{\mathrm{esup}}\limits_{x\in\lambda B}\int_{B^c}J(x,dy)\le\frac{C_\lambda}{W(B)}.$$

\item $(\mathrm{UJS})$, the upper jumping smoothness: the jump kernel $J:M\times M\to\mathbb{R}_+$ exists (that is, the measure $J(x,dy)=J(x,y)d\mu(y)$ for $\mu$-a.e.\ $x\in M$), and there exists $C>0$ such that for $\mu$-a.e.\ $x,y\in M$,
$$J(x,y)\le C\fint_{B(x,r)}J(z,y)d\mu(z)\quad\mbox{with all}\quad 0<r\le\frac{1}{2}d(x,y).$$

\item $(\mathrm{J}_\le)$, the upper bound of jump kernel: the jump kernel $J(x,y)$ exists, and there exists $C>0$ such that
$$J(x,y)\le\frac{C}{V(x,d(x,y))W(x,d(x,y))}\quad\mu\mbox{-a.e.\ }\ x,y\in M.$$
\end{itemize}

This paper involves $(\mathrm{TJ})$ only, which is a relatively weak one among these three conditions since $(\mathrm{J}_\le)\Rightarrow(\mathrm{TJ})$ under $(\mathrm{VD})$ by \cite[Proposition 3.1]{ghh+2}. On the contrary, $(\mathrm{TJ})+(\mathrm{UJS})\Rightarrow(\mathrm{J}_\le)$ by \cite[Proof of Corollary 3.4]{ckw-p}.

Recall also the following conditions on $(\mathcal{E},\mathcal{F})$ (cf.\ \cite{ghh+3,ghh+4}):

\begin{itemize}
\item $(\mathrm{cap}_\le)$, the upper bound of capacity: for any $0<\lambda<1$, there exists $C_\lambda>0$ such that for every ball $B=B(x_0,R)$ with $R\in(0,\overline{R})$, there exists $\phi\in\mathrm{cutoff}(\lambda B,B)$ such that
$$\mathcal{E}(\phi,\phi)\le C_\lambda\frac{\mu(B)}{W(B)}.$$

\item $(\mathrm{Gcap})$, the generalized capacity condition: there exist $C>0$ and $k\ge 1$ such that for all $B_0=B(x,R)$ and $B=B(x,R+r)$ with $0<R<R+r<\overline{R}$, for all $u\in\mathcal{F}'\cap L^\infty$, there exists $\phi\in\mathcal{F}\cap C_0(B)$ such that $0\le\phi\le k$ on $M$, $\phi\ge 1$ on $B_0$ and
$$\mathcal{E}\left(\phi,u^2\phi\right)\le\frac{C}{\inf\limits_{y\in B}W(y,r)}\int_Bu^2d\mu.$$

\item $(\mathrm{ABB})$, the Andres-Barlow-Bass condition (also denoted as $(\mathrm{CSJ})$ in other literature): there exist $C>0$ and $\zeta\ge 0$ such that for all balls $B_0=B(x_0,R),B=B(x_0,R+r)$ and $\Omega=B(x_0,R')$ with $0<R<R+r<R'<\overline{R}$, every $f\in\mathcal{F}'\cap L^\infty$, there exists $\phi\in\mathrm{cutoff}(B_0,B)$ such that
$$\int_\Omega f^2d\Gamma_\Omega(\phi)\le\zeta\int_B\phi^2d\Gamma_B(f)+\frac{C}{\inf\limits_{y\in\Omega}W(y,r)}\int_\Omega f^2d\mu.$$

\item $(\mathrm{PI})$, the (weak) Poincar\'e inequality: there exist $C>0$ and $\lambda_0\in(0,1]$ such that for every ball $B=B(x_0,R)$ with $R\in(0,\overline{R})$ and all $f\in\mathcal{F}'\cap L^\infty$,
$$\fint_{\lambda_0B}\fint_{\lambda_0B}(f(x)-f(y))^2d\mu(x)\mu(y)\le CW(B)\int_Bd\Gamma_B(f).$$

\item $(\mathrm{FK})$, the Faber-Krahn inequality: there exist $C,\nu>0$ and $\sigma\in(0,1)$ such that for every ball $B=B(x_0,R)$ with $R\in(0,\sigma\overline{R})$ and every open subset $U\subset B$,
$$\lambda_1(U):=\inf\limits_{f\in\mathcal{F}(U)\setminus\{0\}}\frac{\mathcal{E}(f,f)}{(f,f)}\ge\frac{C^{-1}}{W(B)}\left(\frac{\mu(B)}{\mu(U)}\right)^\nu.$$

\item $(\mathrm{E})$, the exit time estimate: there exist $C>0$ and $\sigma\in(0,1)$ such that for every ball $B=B(x_0,R)$ with $R\in(0,\overline{R})$, for $\mu$-a.e.\ $y\in B$,
\begin{align}
G^B1_B(y)\le CW(B),&\quad\quad\mbox{if}\quad R<\sigma\overline{R};\label{E-l}\\
G^B1_B(y)\ge C^{-1}W(B),&\quad\quad\mbox{if}\quad y\in\frac{1}{4}B.\label{E-g}
\end{align}
We also say $(\mathrm{E}_\le)$ (or $(\mathrm{E}_\ge)$) holds, if (\ref{E-l}) (or (\ref{E-g})) holds.

\item $(\mathrm{S}^+)$, the strong survival estimate: there exists $C>0$ such that for every ball $B$ and all $t>0$,
$$1-P_t^B1_B\le\frac{Ct}{W(B)}\quad\quad\mbox{on}\quad\frac{1}{4}B.$$
Note that $(\mathrm{S}^+)$ is stronger than the condition $(\mathrm{S}_+)$ in \cite{ghh+3} (where the first term is not ``1'' but a general $\varepsilon\le 1$).

\item $(\mathrm{S}^-)$, the weak survival estimate: there exist $\varepsilon,\delta'>0$ such that for every ball $B=B(x_0,R)$ with $0<R<\overline{R}$ and all $0<t\le\delta'W(B)$,
$$1-P_t^B1_B\le\varepsilon\quad\quad\mbox{on}\quad\frac{1}{4}B.$$

\item $(\mathrm{C})$, conservativeness (or stochastic completeness): $P_t1=1$ for all $t>0$.

\item $(\mathrm{Nash}^+)$, the global Nash inequality: there exist $C,\nu>0$ such that the heat kernel $p_t^B(x,y)$ exists for every ball $B=B(x_0,R)$ with $R\in(0,\overline{R})$, and for all $t>0$,
$$p_t^B(x,y)\le\frac{C}{\mu(B)}\left(\frac{W(B)}{t}\right)^{1/\nu}\exp\left(\frac{Ct}{W(x_0,\overline{R})}\right)\quad\mu\mbox{-a.e.\ }\ x,y\in B.$$

\item $(\mathrm{Nash}^-)$, the local Nash inequality: there exist $C,\nu>0$ and $\sigma\in(0,1)$ such that $p_t^B(x,y)$ exists for every ball $B=B(x_0,R)$ with $R\in(0,\sigma\overline{R})$, and for all $t>0$,
$$p_t^B(x,y)\le\frac{C}{\mu(B)}\left(\frac{W(B)}{t}\right)^{1/\nu}\quad\mu\mbox{-a.e.\ }\ x,y\in B.$$

Obviously $(\mathrm{Nash}^+)$ and $(\mathrm{Nash}^-)$ are identical on an unbounded space (i.e., when $\overline{R}=\infty$).

\item $(\mathrm{LLE})$, the local lower estimate (of the heat kernel) (also denoted as $(\mathrm{NDL})$ in other literature): there exist $\delta_L,c_L\in(0,1)$ such that for every ball $B=B(x_0,R)$ with $R\in(0,\overline{R})$, the heat kernel $p_t^B(x,y)$ exists, and for all $0<t\le W(\delta_LB)$,
$$p_t^B(x,y)\ge\frac{c_L}{V\left(x_0,W^{-1}(x_0,t)\right)}\quad\mu\mbox{-a.e.\ }\ x,y\in B\left(x_0,\delta_LW^{-1}(x_0,t)\right).$$

\item $(\mathrm{PHR})$, the parabolic H\"older regularity: there exist $\delta,\gamma\in(0,1)$ and $C>0$ such that for every ball $B=B(x_0,R)$ with $R\in(0,\overline{R})$, every caloric function $u$ on $Q=(t_0-W(B),t_0]\times B$, all $s,t\in(t_0-W(\delta B),t_0]$ and $\mu$-a.e.\ $x,y\in\delta B$,
\begin{equation}\label{phr}
|u(t,x)-u(s,y)|\le C\left(\frac{d(x,y)\vee W^{-1}(y,|s-t|)}{R}\right)^\gamma\left(\|u\|_Q+W(B)^{1-\frac{1}{p}}T_Q^{(\frac{2}{3};p)}(u)\right).
\end{equation}
We say $(\mathrm{PHR}^0)$ holds, if (\ref{phr}) holds with $\|u\|_Q+W(B)^{1-\frac{1}{p}}T_Q^{(\frac{2}{3};p)}(u)$ replaced by $\|u\|_{(t_0-W(B),t_0]\times M}$.

\item $(\mathrm{EHR})$, the elliptic H\"older regularity: there exist $\delta,\gamma\in(0,1)$ and $C>0$ such that for every ball $B=B(x_0,R)$ with $R\in(0,\overline{R})$, every harmonic function $f$ on $B$ and $\mu$-a.e.\ $x,y\in\delta B$,
\begin{equation}\label{ehr}
|f(x)-f(y)|\le C\left(\frac{d(x,y)}{R}\right)^\gamma\left(\|f\|_{L^\infty(B)}+W(B)T_B^{(2/3)}(f)\right).
\end{equation}
We say $(\mathrm{EHR}^0)$ holds, if (\ref{ehr}) holds with $\|f\|_{L^\infty(B)}+W(B)T_B^{(2/3)}(f)$ replaced by $\|f\|_{L^\infty(M)}$.
\end{itemize}

Clearly $(\mathrm{S}^+)\Rightarrow(\mathrm{S}^-)$, $(\mathrm{Gcap})\Rightarrow(\mathrm{cap}_\le)$, $(\mathrm{PHR}^+)\Rightarrow(\mathrm{EHR}^+)$ and $(\mathrm{PHR})\Rightarrow(\mathrm{EHR})$. Further, since $(\mathrm{VD})$ and $(\mathrm{RVD})$ hold, recall the following implications:
\begin{align}
(\mathrm{PI})\Rightarrow(\mathrm{Nash}^+)\Rightarrow(\mathrm{FK})&\ \ \mbox{by \cite[Lemma 4.11, Propositions 4.7 and 4.13]{ghh+3}}\label{4-11}\\
(\mathrm{LLE})\Rightarrow(\mathrm{PI})+(\mathrm{S}^-)&\ \ \mbox{by \cite[Proposition 4.7, Lemmas 7.12 and 7.14]{ghh+3}}\label{vLE}\\
(\mathrm{E})\Rightarrow(\mathrm{S}^-)&\ \ \mbox{by \cite[Proposition 13.4]{ghh+1}}\label{13-4}
\end{align}

Note that $(\mathrm{S}^-)$ is weaker than the condition $(\mathrm{S})$ in \cite{ghh+3} (where $\varepsilon$ could be arbitrarily taken), but the properties like \cite[Corollary 8.9]{ghh+3} and \cite[Lemma 13.5]{ghh+1} still hold (according to the proofs themselves). Combining \cite[Notes on Page 118]{ghh+1}, there is
\begin{equation}\label{s-c}
(\mathrm{S}^-)\Rightarrow(\mathrm{C})+(\mathrm{Gcap}).
\end{equation}

With $0<q\le 1$, define the following Harnack inequalities and growth lemmas (meanwhile, we say $(\bullet)$ holds, if $(\bullet_q)$ holds with some $0<q\le 1$):

\begin{itemize}
\item $(\mathrm{PHI})$, the parabolic Harnack inequality: there exist $\delta>0$, $C>1$ and $\lambda_0\in(0,1)$ such that for every ball $B=B(x_0,R)$ with $R\in(0,\overline{R})$ and all $u:(t_0-4\delta W(B),t_0]\to\mathcal{F}'(B)$ that is non-negative, bounded and caloric on $Q=(t_0-4\delta W(B),t_0]\times B$,
\begin{equation}\label{phi}
\mathop{\mathrm{esup}}_{(t_0-3\delta W(B),t_0-2\delta W(B)]\times\lambda_0B}u\le C\left(\mathop{\mathrm{einf}}_{(t_0-\delta W(B),t_0]\times\lambda_0B}u+\overline{T}_Q^{(3/4)}(u_-)\right).
\end{equation}
We say $(\mathrm{PHI}^0)$ holds, if (\ref{phi}) holds only for $u$ that is non-negative on $(t_0-4\delta W(B),t_0]\times M$.

\item $(\mathrm{EHI})$, the elliptic Harnack inequality: there exist $C>1$ and $\lambda_0\in(0,1)$ such that for every ball $B=B(x_0,R)$ with $R\in(0,\overline{R})$ and all $f\in\mathcal{F}'(B)$ that is non-negative, bounded and harmonic on $B$,
\begin{equation}\label{ehi}
\mathop{\mathrm{esup}}_{\lambda_0B}f\le C\left(\mathop{\mathrm{einf}}_{\lambda_0B}f+W(B)T_B^{(3/4)}(f_-)\right).
\end{equation}
We say $(\mathrm{EHI}^0)$ holds, if (\ref{ehi}) holds only for $f$ that is non-negative on $M$.

\item $(\mathrm{wPH}_q)$, the $q$-weak parabolic Harnack inequality: there exist $\delta>0$, $C>1$ and $\lambda_0\in(0,1)$ such that for every ball $B=B(x_0,R)$ with $R\in(0,\overline{R})$, all $0<\lambda\le\lambda_0$ and all $u$ that is non-negative, bounded and supercaloric on $Q=(t_0-4\delta W(\lambda B),t_0]\times B$,
\begin{equation}\label{wph'q}
\left(\fint_{(t_0-3\delta W(\lambda B),t_0-2\delta W(\lambda B)]\times\lambda B}u^qd\hat{\mu}\right)^{1/q}\le C\left(\mathop{\mathrm{einf}}_{(t_0-\delta W(\lambda B),t_0]\times\lambda B}u+\overline{T}_Q^{(2/3)}(u_-)\right).
\end{equation}
We say $(\mathrm{wPH}_q^0)$ holds, if (\ref{wph'q}) holds only for $u$ that is non-negative on $(t_0-4\delta W(B),t_0]\times M$; say $(\mathrm{wPH}_q^-)$ holds, if it only holds for caloric and globally non-negative $u$.

\item $(\mathrm{wEH}_q)$, the $q$-weak elliptic Harnack inequality: there exist $C>1$ and $\lambda_0\in(0,1)$ such that for every ball $B=B(x_0,R)$ with $R\in(0,\overline{R})$, every $0<\lambda\le\lambda_0$ and all non-negative, bounded and superharmonic $f$ on $B$,
\begin{equation}\label{weh'q}
\left(\fint_{\lambda B}f^qd\mu\right)^{1/q}\le C\left(\mathop{\mathrm{einf}}_{\lambda B}f+W(\lambda B)T_B^{(2/3)}(f_-)\right).
\end{equation}
Define properties $(\mathrm{wEH}_q^0)$ and $(\mathrm{wEH}_q^-)$ similarly as in the parabolic case.

\item $(\mathrm{PGL}_q)$, the $q$-parabolic growth lemma: there exist $\delta>0$, $C>1$ and $\lambda_0\in(0,1)$ such that for every ball $B=B(x_0,R)$ with $R\in(0,\overline{R})$, all $0<\lambda\le\lambda_0$ and all $u$ that is non-negative, bounded and supercaloric on $Q=(t_0-4\delta W(\lambda B),t_0]\times B$, if
$$\frac{\hat{\mu}\big(Q_-\cap\{u\ge a\}\big)}{\hat{\mu}(Q_-)}\ge\eta$$
with some $a,\eta>0$, where $Q_-=(t_0-3\delta W(\lambda B),t_0-2\delta W(\lambda B)]\times\lambda B$, then
$$\mathop{\mathrm{einf}}_{(t_0-\delta W(\lambda B),t_0]\times\lambda B}u\ge C^{-1}\eta^{1/q}a-2\overline{T}_Q^{(2/3)}(u_-).$$

\item $(\mathrm{EGL}_q)$, the $q$-elliptic growth lemma: there exist $C>1$ and $\lambda_0\in(0,1)$ such that for every ball $B=B(x_0,R)$ with $R\in(0,\overline{R})$, all $0<\lambda\le\lambda_0$ and $f$ that is non-negative, bounded and superharmonic on $B$, if
$$\frac{\mu\big(\lambda B\cap\{f\ge a\}\big)}{\mu(\lambda B)}\ge\eta$$
with some $a,\eta>0$, then
$$\mathop{\mathrm{einf}}_{\lambda B}f\ge C^{-1}\eta^{1/q}a-W(\lambda B)T_B^{(2/3)}(f_-).$$
\end{itemize}

Clearly $(\mathrm{PHI})\Rightarrow(\mathrm{EHI})$, $(\mathrm{wPH}_q)\Rightarrow(\mathrm{wEH}_q)$, $(\mathrm{PGL}_q)\Rightarrow(\mathrm{EGL}_q)$, and $(\bullet_q)\Rightarrow(\bullet_{q'})$ for all $0<q'<q\le 1$. Note that the ``lemma of growth'' $(\mathrm{LG})$ in \cite{ghh+1} is weaker than $(\mathrm{EGL}_q)$ with any $q>0$.

Trivially $(\bullet)\Rightarrow(\bullet^0)$ and $(\bullet^0)\Rightarrow(\bullet^-)$. On the converse, we show in Propositions \ref{wph+t}, \ref{phi+t} and \ref{ehi+t} that $(\bullet^0)\Rightarrow(\bullet)$ under proper assumptions. Note that in $(\bullet^0)$ and $(\bullet^-)$, the tail terms vanish since $u_-$ or $f_-$ is constantly $0$.

Now we introduce our main theorems:

\begin{theorem}\label{M1}
Assume that $(M,d,\mu)$ satisfies $(\mathrm{VD})$ and $(\mathrm{RVD})$, and $(\mathcal{E},\mathcal{F})$ is a regular Dirichlet form such that $dj(x,y)=J(x,dy)d\mu(x)$ and $\kappa\equiv 0$. Then the following are equivalent, each implying $(\mathrm{S}^-)$:
$$\mathrm{(i)}\ \ (\mathrm{wPH}_1)\quad\quad\mathrm{(ii)}\ \ (\mathrm{wPH})+(\mathrm{Nash}^-)\quad\quad\mathrm{(iii)}\ \ (\mathrm{PGL})+(\mathrm{FK})\quad\quad\mathrm{(iv)}\ \ (\mathrm{LLE}).$$
\end{theorem}

\begin{proof}
Since $(\mathrm{wPH})\Rightarrow(\mathrm{S}^-)$ by Lemma \ref{s-}, it suffices to prove the equivalences.

On one hand, $\mathrm{(iv)}\Rightarrow\mathrm{(i)}+\mathrm{(ii)}$ by Proposition \ref{w_D} and (\ref{vLE}); on the other hand, $\mathrm{(i)}\Rightarrow\mathrm{(iv)}$ and $\mathrm{(ii)}\Rightarrow\mathrm{(iv)}$ by Proposition \ref{lle}. Thus $\mathrm{(i)}\Leftrightarrow\mathrm{(ii)}\Leftrightarrow\mathrm{(iv)}$.

Finally, $\mathrm{(ii)}\Leftrightarrow\mathrm{(iii)}$ by Proposition \ref{wph} and Lemma \ref{nash}, completing the proof.
\end{proof}

Indeed, by Lemma \ref{nash}, $(\mathrm{PGL})+(\mathrm{Nash}^-)$ and $(\mathrm{wPH})+(\mathrm{FK})$ are also equivalent to these conditions.

\begin{theorem}\label{M2}
Under the same assumptions as in Theorem \ref{M1}, if $(\mathrm{TJ})$ holds, then the conditions $(\mathrm{i})$--$(\mathrm{iv})$ are also equivalent to each of the following conditions:
\begin{longtable}{clcclccl}
$\mathrm{(i)'}$&$(\mathrm{wEH})+(\mathrm{cap}_\le)+(\mathrm{FK})$&\quad&$\mathrm{(ii)'}$&$(\mathrm{PHR})+(\mathrm{E})$&\quad&$\mathrm{(iii)'}$&$(\mathrm{PI})+(\mathrm{ABB})$\\
$\mathrm{(i)''}$&$(\mathrm{wEH})+(\mathrm{E})$&&$\mathrm{(ii)''}$&$(\mathrm{EHR}^0)+(\mathrm{E})$&&$\mathrm{(iii)''}$&$(\mathrm{PI})+(\mathrm{Gcap})$
\end{longtable}
\end{theorem}

Trivially $(\mathrm{wPH})+(\mathrm{E})$, $(\mathrm{EHR})+(\mathrm{E})$ and $(\mathrm{PHR}^0)+(\mathrm{E})$ are also equivalent to these conditions.

\begin{proof}
$(\mathrm{wPH})+(\mathrm{S}^-)+(\mathrm{FK})\Rightarrow\mathrm{(i)'}$ by (\ref{s-c}); under $(\mathrm{TJ})$, $\mathrm{(i)'}\Rightarrow\mathrm{(i)''}$ by \cite[Lemma 6.3]{hy}; $\mathrm{(i)''}\Rightarrow\mathrm{(ii)''}$ by \cite[Lemma 6.2]{hy}; $\mathrm{(ii)''}\Rightarrow\mathrm{(iv)}$ by Proposition \ref{e-lle}. Combining Theorem \ref{M1}, we obtain $\mathrm{(ii)}\Leftrightarrow\mathrm{(i)'}\Leftrightarrow\mathrm{(i)''}\Leftrightarrow\mathrm{(ii)''}$.

$\mathrm{(iii)}+\mathrm{(iv)}\Rightarrow\mathrm{(ii)'}$ by Lemma \ref{L1-E} and Proposition \ref{v_D}; $\mathrm{(ii)'}\Rightarrow\mathrm{(ii)''}$ by Lemma \ref{in-hr} and the trivial implication $(\mathrm{PHR}^0)\Rightarrow(\mathrm{EHR}^0)$. Thus $\mathrm{(ii)''}$ is also equivalent to the conditions above.

Finally, $\mathrm{(iv)}\Leftrightarrow\mathrm{(iii)'}\Leftrightarrow\mathrm{(iii)''}$ is given by \cite[Theorem 2.10]{ghh+3}.
\end{proof}

Let us compare these results with the known ones (see \cite[Theorem 1.20]{ckw-e} and \cite[Theorem 1.18]{ckw-p} for the case $\mathcal{E}^{(L)}\equiv 0$; see \cite[Theorem 2.15]{bbk}, \cite[Theorem 3.1]{bgk} and \cite[Theorem 1.2]{ghl} for the case $J\equiv 0$):

\begin{proposition}\label{1-18}
Assume that $\overline{R}=\infty$ and $W(x,r)=\phi(r)$ with some function $\phi$. For $X$-caloric functions,
\begin{itemize}
\item[$(\mathrm{i})$] if $\mathcal{E}^{(L)}\equiv 0$, then
\begin{align*}
(\mathrm{PHI}^0)\Leftrightarrow\ &(\mathrm{EHI}^0)+(\mathrm{E})+(\mathrm{UJS})\Leftrightarrow(\mathrm{PHR}^0)+(\mathrm{E})+(\mathrm{UJS})\Leftrightarrow(\mathrm{EHR}^0)+(\mathrm{E})+(\mathrm{UJS})\\
\Leftrightarrow\ &(\mathrm{LLE})+(\mathrm{UJS})\Leftrightarrow(\mathrm{PI})+(\mathrm{ABB})+(\mathrm{J}_\le)+(\mathrm{UJS});
\end{align*}
\item[$(\mathrm{ii})$] if $J\equiv 0$, then $(\mathrm{PHI}^0)\Leftrightarrow(\mathrm{EHI}^0)+(\mathrm{E})\Leftrightarrow(\mathrm{LLE})\Leftrightarrow(\mathrm{PI})+(\mathrm{ABB})$.
\end{itemize}
\end{proposition}

Clearly Theorems \ref{M1} and \ref{M2} serve as a generalization of this proposition, with the weak assumption $(\mathrm{TJ})$ replacing the strong $(\mathrm{UJS})$, and weak Harnack inequalities replacing the strong ones.

On the other hand, let us recall a counterexample collected in \cite[Example 5.4]{ckw-p}, showing that our equivalent conditions can not contain any kind of strong Harnack inequality:

\begin{example}\label{BC}
Let $(M,d,\mu)=(\mathbb{R}^3,|\cdot|,m_3)$ be the 3-dimensional Euclidean space with Lebesgue measure. For any $0<\beta<2$, the jump form defined by
$$J(x,dy)=\frac{\delta_{x_1}(dy_1)\delta_{x_2}(dy_2)dy_3+\delta_{x_2}(dy_2)\delta_{x_3}(dy_3)dy_1+\delta_{x_3}(dy_3)\delta_{x_1}(dy_1)dy_2}{|x-y|^{1+\beta}}$$
for all $x=(x_1,x_2,x_3)$ and $y=(y_1,y_2,y_3)$ in $\mathbb{R}^3$. is regular. Clearly $(\mathrm{TJ})$ holds. Further, $(\mathrm{E})$ holds by \cite[Proposition 2.1]{bc}, and $(\mathrm{EHR}^0)$ holds by \cite[Theorem 2.9]{bc}. Therefore, all the equivalent conditions $\mathrm{(i)}$ -- $\mathrm{(iv)}$, $\mathrm{(i)'}$ -- $\mathrm{(iii)'}$ and $\mathrm{(i)''}$ -- $\mathrm{(iii)''}$ hold. Meanwhile, $(\mathrm{EHI}^0)$ fails according to \cite[Pages 501-502]{bc}.
\end{example}

A similar construction gives an example of a mixed form with scale $t=r^2$. Recall that on the Sierpi\'nski gasket $\mathcal{SG}\subset\mathbb{R}^2$, the $\log_35$-dimensional Hausdorff measure $m_\ast$ is Ahlfors regular, and for each $0<\beta<\log_25$, a jump kernel $J(x,y)\asymp d(x,y)^{-\log_35-\beta}$ defines a (stable-like) regular Dirichlet form.

\begin{example}
Consider $M=[0,1]\times\mathcal{SG}$ with metric measure structure
$$d(x,y)=\max\big\{|x_1-y_1|,|x_2-y_2|\big\},\quad d\mu(x)=dx_1dm_\ast(x_2)$$
for all $x=(x_1,x_2),y=(y_1,y_2)\in M$. Let
$$\mathcal{E}(f,f)=\int_M\left|\frac{\partial f}{\partial x_1}\right|^2d\mu+\int_M\int_{\mathcal{SG}}\frac{(f(x_1,x_2)-f(x_1,y_2))^2}{|x_2-y_2|^{\log_35+2}}dm_\ast(y_2)d\mu(x_1,x_2),$$
so that it contains non-trivial strongly local and jump parts. Using the same method as for Example \ref{BC} we could see $(\mathrm{TJ})$, $(\mathrm{E})$ and $(\mathrm{EHR}^0)$ hold with the same scale. Hence all the equivalent conditions in Theorems \ref{M1} and \ref{M2} hold.
\end{example}

Note that the assumptions $\mathcal{E}^{(L)}\equiv 0$ and $\overline{R}=\infty$ are not essential in \cite{ckw-e,ckw-p} (by the proofs themselves); while if $W$ depends on $x$, then we could obtain the same results by taking a new metric $d_\ast$ (see Section \ref{ctn}) so that all the involved conditions are equivalently expressed under $d_\ast$ with scale $\phi(r)=r^\beta$. Further, the analytic $(\mathrm{PHI}^0)$ follows from the stochastic one with the help of Propositions \ref{sp-c} and \ref{AIS}, so that $(\mathrm{LLE})+(\mathrm{UJS})\Rightarrow(\mathrm{PHI}^0)$ holds in the analytic sense accordingly to Proposition \ref{1-18}. This implication can also be proved directly by analysis, for example, by reproducing the proof of \cite[Theorem 4.3]{ckw-p} in an analytic manner (based on the constructions $u_s^{U;\Omega}$ and $u_{t_1,U}$ below). Alternatively, we could follow Moser's original idea, that is, to combine $(\mathrm{wPH})$ with a parabolic mean-value inequality that sharpens \cite[Theorem 2.10]{ghh+2} under condition $(\mathrm{UJS})$. On the converse, the proof of $(\mathrm{PHI}^0)\Rightarrow(\mathrm{LLE})+(\mathrm{UJS})$ in \cite[Section 3]{ckw-p} involves only the heat kernel and a caloric extension (see Section \ref{Tail}). Hence it is not hard to show Proposition \ref{1-18} in pure analysis in our general setting. However, these steps are technically very complicated, for which reason we would split this part into a subsequent paper.

The following consequence is especially interesting:

\begin{corollary}\label{M4}
Under the same assumptions,
$$(\mathrm{PHI})\Leftrightarrow(\mathrm{wPH}_1)+(\mathrm{UJS})\Leftrightarrow(\mathrm{wPH})+(\mathrm{FK})+(\mathrm{UJS}).$$
\end{corollary}

Since $(\mathrm{UJS})$ is trivial in the following cases:
\begin{itemize}
\item[$\mathrm{i)}$] any strongly local form. In particular, an arbitrary second-order uniformly elliptic operator of divergence form on a Riemannian manifold;

\item[$\mathrm{ii)}$] any subordinated (sub-)Gaussian diffusion with any L\'evy measure (see \cite[Theorem 2.3]{lm2}). In particular, any $\alpha$-stable process with $0<\alpha<2$ (for example, $(-\Delta)^{-\alpha/2}$) on $\mathbb{R}^d$,
\end{itemize}
then $(\mathrm{wPH}_1)$ is actually the same as $(\mathrm{PHI})$, while $(\mathrm{wPH})$ differs only by $(\mathrm{FK})$. Note that $(\mathrm{FK})$ is equivalent to the Sobolev embedding theorem in Case $\mathrm{i)}$ by \cite[Exercises in Section 14.3]{gri2}, while Case $\mathrm{ii)}$ is similar (with respect to Besov spaces) with the help of Lemma \ref{nash} below.

In the sequel, $C$ always stands for a positive constant (independent of the variables in the context but possibly depending on parameters), whose value may change from line to line, but is typically greater than 1. When necessary, we refer to the constant $C$ coming from condition $(\bullet)$ by $C_\bullet$ so that its value is fixed once appearing.

\section{Proof of Theorem \ref{M1}}\label{PM1}

To begin with, we give an equivalent expression of $(\mathrm{LLE})$, which is much easier to verify:

\begin{lemma}\label{LE1}
$(\mathrm{LLE})$ is equivalent to the following condition $(\mathrm{LLE}^-)$: there exist $\delta>0$ and $c,\sigma\in(0,1)$ such that for every ball $B=B(x_0,R)$ with $R\in(0,\sigma\overline{R})$ and all $0<t\le W(\delta B)$,
\begin{equation}\label{wLL}
p_t^B(x,dy)\ge\frac{cd\mu(y)}{V\left(x_0,W^{-1}(x_0,t)\right)}\quad\mu\mbox{-a.e.\ }\ x\in B'
\end{equation}
as a measure on $B'$, where $B':=B(x_0,\delta W^{-1}(x_0,t))$.
\end{lemma}

\begin{proof}
Trivially $(\mathrm{LLE})\Rightarrow(\mathrm{LLE}^-)$. On the converse, assume that $(\mathrm{LLE}^-)$ holds, and set $\delta_L=\delta\sigma$. Then, for every ball $B=B(x_0,R)$ with $0<R<\overline{R}$ and all $0<t\le W(\delta_LB)$, (\ref{wLL}) holds for $B$ by applying $(\mathrm{LLE}^-)$ to $\sigma B$ and using the fact that $p_t^B(x,dy)\ge p_t^{\sigma B}(x,dy)$.

It is easy to obtain $(\mathrm{PI})$ from (\ref{wLL}) by repeating the proof of \cite[Lemma 7.12]{ghh+3}, but replacing $p_t^B(x,y)(u(x)-u(y))^2d\mu(y)$ on the ninth line of ``Case 1'' there by $(u(x)-u(y))^2p_t^B(x,dy)$. It follows by (\ref{4-11}) that $(\mathrm{Nash}^+)$ holds. In particular, $p_t^B(x,y)$ exists. Again by (\ref{wLL}), we have
$$p_t^B(x,y)\ge\frac{c}{V\left(x_0,W^{-1}(x_0,t)\right)}$$
for $\mu$-a.e.\ $x,y\in B'=B(x_0,\delta_LW^{-1}(x_0,t))$, which is exactly the condition $(\mathrm{LLE})$.
\end{proof}

\begin{lemma}\label{L4-3}
Assume that $\Omega$ is a bounded open set in $M$, and $\chi:\mathbb{R}\to\mathbb{R}$ is a non-decreasing function such that $\chi(s)=0$ for all $s\le 0$ and $Lip(\chi)\le 1$. Then, for every $f\in\mathcal{F}'(\Omega)$ such that $f\le 0$ $\mu$-a.e.\ on $\Omega^c$, there is
\begin{equation}\label{4-3}
\chi(f)\in\mathcal{F}(\Omega)\quad\mbox{and}\quad\tilde{\mathcal{E}}(f,\chi(f))\ge\mathcal{E}(\chi(f),\chi(f)).
\end{equation}
\end{lemma}

\begin{proof}
By the definition of $\mathcal{F}'(\Omega)$, there exists $w\in\mathcal{F}$ such that $f=w$ on some $\Omega'\Supset\Omega$. Without loss of generality, we assume $f\le w\le 0$ on $\Omega^c$ (otherwise we replace $\Omega'$ by some $\Omega''$ with $\Omega\Subset\Omega''\Subset\Omega'$, and then replace $w$ by $w'=w\varphi_0\in\mathcal{F}$ with some $\varphi_0\in\mathrm{cutoff}(\Omega'',\Omega')$). This way, $\chi(f)=\chi(w)$ directly.

Since (\ref{4-3}) holds for $w$ by \cite[Lemma 4.3]{gh08}, then $\chi(f)\in\mathcal{F}(\Omega)$. Further, since $w-f$ is supported outside $\Omega'$, and $\chi(w)$ is supported in $\Omega$, both non-negative, then
\begin{align*}
\tilde{\mathcal{E}}(f,\chi(f))=&\ \tilde{\mathcal{E}}(f,\chi(w))=\mathcal{E}(w,\chi(w))+\tilde{\mathcal{E}}(f-w,\chi(w))\\
=&\ \mathcal{E}(w,\chi(w))+\mathcal{E}^{(L)}(f-w,\chi(w))+\int_{M\times M}\big(\chi(w)(x)-\chi(w)(y)\big)\big((f-w)(x)-(f-w)(y)\big)dj(x,y)\\
=&\ \mathcal{E}(w,\chi(w))+0+2\int_{\Omega\times(\Omega')^c}\chi(w)(x)(w(y)-f(y))dj(x,y)\ge\mathcal{E}(\chi(w),\chi(w))+0=\mathcal{E}(\chi(f),\chi(f)),
\end{align*}
showing (\ref{4-3}) exactly for $f$.
\end{proof}

Below is a generalized version of the parabolic maximum principle:

\begin{proposition}\label{gpmp}
Let $\Omega$ be a bounded open set in $M$, and $Q=(t_1,t_2]\times\Omega$. If
\begin{itemize}
\item $u:(t_1,t_2]\to\mathcal{F}'(\Omega)$ is subcaloric on $Q$;
\item $u\le 0$ weakly on $(t_1,t_2]\times\Omega^c$, that is, $u_+(t,\cdot)\in\mathcal{F}(\Omega)$ for all $t_1<t\le t_2$;
\item $u|_{t=t_1}\le 0$ weakly in $\Omega$, that is, $\|u_+(t,\cdot)\|_{L^2(\Omega)}\to 0$ as $t\downarrow t_1$,
\end{itemize}
then $u\le 0$ on $(t_1,t_2]\times M$.
\end{proposition}

\begin{proof}
Repeat the proof of \cite[Proposition 4.11]{gh08} on each interval $(t_1,s_1]$ and $(s_i,s_{i+1}]$, but using Lemma \ref{L4-3} whenever \cite[Lemma 4.3]{gh08} was used there.
\end{proof}

The following lemma is extremely important in this paper:

\begin{lemma}\label{comp}
Given any open sets $U\Subset\Omega$, for every $u:(t_1,t_2]\to\mathcal{F}'(\Omega)$ that is supercaloric, non-negative on $Q=(t_1,t_2]\times\Omega$ with $\overline{T}_Q^U(u_-)<\infty$, for all $t_1<s<t\le t_2$ and $\mu$-a.e.\ $x\in U$,
\begin{equation}\label{u-l}
u(t,x)\ge P_{t-s}^Uu(s,\cdot)(x)-2\overline{T}_{(s,t]\times\Omega}^U(u_-).
\end{equation}
\end{lemma}

\begin{proof}
For all $t_1<s<t\le t_2$ and $x\in M$, define
$$u_s^{U;\Omega}(t,x)=u_+(t,x)+2\overline{T}_{(s,t]\times\Omega}^U(u_-).$$
Clearly $u_s^{U;\Omega}\ge 0$ on $(s,t_2]\times M$. Since $u_+=u$ on $\mathrm{supp}(\varphi)$, for every non-negative $\varphi\in\mathcal{F}(U)$ we have
\begin{align*}
\frac{\mathrm{d}}{\mathrm{d}t}\left(u_s^{U;\Omega},\varphi\right)+\mathcal{E}\left(u_s^{U;\Omega},\varphi\right)=&\ 2\left(T_\Omega^U(u_-;t),\varphi\right)+\frac{\mathrm{d}}{\mathrm{d}t}(u_+,\varphi)+\mathcal{E}(u_+,\varphi)\\
=&\ 2\left(T_\Omega^U(u_-;t),\varphi\right)_{L^2(U)}+\frac{\mathrm{d}}{\mathrm{d}t}(u,\varphi)+\big(\mathcal{E}(u,\varphi)+\mathcal{E}(u_-,\varphi)\big)\\
\ge&\ 2\left(T_\Omega^U(u_-;t),\varphi\right)_{L^2(U)}-2\int_U\int_{\Omega^c}\varphi(x)u_-(t,y)J(x,dy)d\mu(x)\\
\ge&\ 2\left(T_\Omega^U(u_-;t),\varphi\right)_{L^2(U)}-2\int_U\varphi(x)\left(\int_{\Omega^c}\sup\limits_{x'\in U}u_-(t,y)J(x',dy)\right)d\mu(x)=0.
\end{align*}
Hence $u_s^{U;\Omega}$ is supercaloric on $Q_s:=(s,t_2]\times U$, that is, $v(t,\cdot):=P_{t-s}^Uu(s,\cdot)-u_s^{U;\Omega}(t,\cdot)$ is subcaloric on $Q_s$.

Clearly $v\le 0$ on $U^c$, while
$$\lim\limits_{s'\downarrow s}v(s',\cdot)=u(s,\cdot)-u_+(s,\cdot)=0$$
in $L^2(U)$. Thus $v\le 0$ on $Q_s$ by Proposition \ref{gpmp}. In particular,
$$u(t,x)=u_+(t,x)=u_s^{U;\Omega}(t,x)-2\overline{T}_{(s,t]\times\Omega}^U(u_-)\ge P_{t-s}^Uu(s,\cdot)(x)-2\overline{T}_{(s,t]\times\Omega}^U(u_-)$$
for all $t\in(s,t_2]$ and $\mu$-a.e.\ $x\in\lambda B$, which completes the proof.
\end{proof}

\begin{proposition}\label{w_D}
$(\mathrm{LLE})\Rightarrow(\mathrm{wPH}_1)$.
\end{proposition}

\begin{proof}
Let $\delta>0$ be an arbitrary number. By (\ref{u-l}) with $\lambda=\frac{2}{3}$, we obtain for every $u$ that is non-negative and supercaloric on $(t_0-4\delta W(B),t_0]\times B$, all $t_0-4\delta W(B)<s<t\le t_0$ and $\mu$-a.e.\ $x\in\frac{2}{3}B$ that
$$u(t,x)\ge P_{t-s}^{\frac{2}{3}B}u(s,\cdot)(x)-2\overline{T}_{(s,t]\times B}^{(2/3)}(u_-)=\int_{\frac{2}{3}B}p_{t-s}^{\frac{2}{3}B}(x,y)u(s,y)d\mu(y)-2\overline{T}_{(s,t]\times B}^{(2/3)}(u_-).$$
Integrating this inequality over $s\in(t_0-3\delta W(\lambda'B),t_0-2\delta W(\lambda'B)]$, we see
\begin{equation}\label{epd-l}
	u(t,x)\ge\frac{1}{\delta W(\lambda'B)}\int_{t_0-3\delta W(\lambda'B)}^{t_0-2\delta W(\lambda'B)}\int_{\frac{2}{3}B}p_{t-s}^{\frac{2}{3}B}(x,y)u(s,y)d\mu(y)-2\overline{T}_Q^{(2/3)}(u_-)
\end{equation}
for all $t_0-\delta W(\lambda'B)<t\le t_0$ and $x\in\lambda'B$, with an arbitrary $\lambda'\in(0,1)$.

Now we select
$$\delta=C_W\delta_L^{-\beta_2},\quad \lambda_0=\frac{2}{3}\cdot\frac{\delta_L}{(3C_w\delta)^{1/\beta_1}}.$$
For all $(s,y)\in Q_-:=(t_0-3\delta W(\lambda'B),t_0-2\delta W(\lambda'B)]\times\lambda'B$ and $(t,x)\in(t_0-\delta W(\lambda'B),t_0]\times\lambda'B$, where $0<\lambda'\le\lambda_0$, it follows that $W(\delta_L^{-1}\lambda'B)\le\delta W(\lambda'B)<t-s$ and
$$t-s<3\delta W(\lambda'B)\le 3\delta C_W\left(\frac{\lambda'}{\frac{2}{3}\delta_L}\right)^{\beta_1}W\left(\delta_L\cdot\frac{2}{3}B\right)\le W\left(\delta_L\cdot\frac{2}{3}B\right).$$
Hence by $(\mathrm{LLE})$,
$$p_{t-s}^{\frac{2}{3}B}(x,y)\ge\frac{c_L}{V\left(x_0,W^{-1}(x_0,t-s)\right)}\ge\frac{c_L}{V\left(x_0,W^{-1}(x_0,3\delta W(\lambda'B))\right)}\ge\frac{c(\delta)}{\mu(\lambda'B)}.$$
Inserting this inequality into (\ref{epd-l}), we obtain
$$u(t,x)+2\overline{T}_Q^{(2/3)}(u_-)\ge\frac{1}{\delta W(\lambda'B)}\int_{Q_-}\frac{c(\delta)u(s,y)}{\mu(\lambda'B)}d\hat{\mu}(s,y)=\frac{c(\delta)}{\hat{\mu}(Q_-)}\int_{Q_-}ud\hat{\mu},$$
which is exactly $(\mathrm{wPH}_1)$.
\end{proof}

\begin{proposition}\label{wph}
$(\mathrm{wPH}_q)\Rightarrow(\mathrm{PGL}_q)$, and $(\mathrm{PGL}_q)\Rightarrow(\mathrm{wPH}_{q'})$ for all $0<q'<q\le 1$.
\end{proposition}

\begin{proof}
(1) Assume that $(\mathrm{wPH}_q)$ holds. If $\hat{\mu}(Q_-\cap\{u\ge a\})\ge\eta\hat{\mu}(Q_-)$, then
$$a\eta^{1/q}\le\left(\fint_{Q_-}u^qd\hat{\mu}\right)^{1/q}\le C_{\mathrm{wPH}}\left(\mathop{\mathrm{einf}}_{(t_0-\delta W(\lambda B),t_0]\times\lambda B}u+\overline{T}_Q^{(2/3)}(u_-)\right)$$
by the Markov inequality, showing that $(\mathrm{PGL}_q)$ holds.

(2) Assume that $(\mathrm{PGL}_q)$ holds. Set $Q_-=(t_0-3\delta W(\lambda B),t_0-2\delta W(\lambda B)]\times\lambda B$. For each $a>0$ and $0<\lambda<\lambda_0$, by taking $\eta=\hat{\mu}\big(Q_-\cap\{u\ge a\}\big)/\hat{\mu}(Q_-)$ we obtain
$$A_u:=\mathop{\mathrm{einf}}_{(t_0-\delta W(\lambda B),t_0]\times\lambda B}u+W(\lambda B)T_Q^{(2/3)}(u_-)\ge C_{\mathrm{PGL}}^{-1}\left(\frac{\hat{\mu}\big(Q_-\cap\{u\ge a\}\big)}{\hat{\mu}(Q_-)}\right)^{1/q}a.$$
Therefore, for any $0<q'<q$,
\begin{align*}
\fint_{Q_-}u^{q'}d\hat{\mu}=&\ \int_0^\infty q'a^{q'-1}\frac{\hat{\mu}\big(Q_-\cap\{u\ge a\}\big)}{\hat{\mu}(Q_-)}da=\left(\int_0^{A_u}+\int_{A_u}^\infty\right)q'a^{q'-1}\frac{\hat{\mu}\big(Q_-\cap\{u\ge a\}\big)}{\hat{\mu}(Q_-)}da\\
\le&\ \int_0^{A_u}q'a^{q'-1}da+\int_{A_u}^\infty q'a^{q'-1}\left(C_{\mathrm{PGL}}A_ua^{-1}\right)^qda=\left(1+\frac{q'C_{\mathrm{PGL}}^q}{q-q'}\right)A_u^{q'},
\end{align*}
which is exactly $(\mathrm{wPH}_{q'})$.
\end{proof}

Note that the elliptic analog, that is, $(\mathrm{wEH}_q)\Rightarrow(\mathrm{ELG}_q)\Rightarrow(\mathrm{wEH}_{q'})$ for all $0<q'<q\le 1$, is essentially contained in the proof of \cite[Proposition 5.2]{hy}.

\begin{lemma}\label{nash}
$(\mathrm{FK})\Leftrightarrow(\mathrm{Nash}^-)$.
\end{lemma}

\begin{proof}
The implication ``$\Rightarrow$'' follows from \cite[Lemmas 5.4 and 5.5]{gh14} by taking $a=\left(C_{\mathrm{FK}}W(B)\right)^{-1}\mu(B)^\nu$.

Now we prove ``$\Leftarrow$''. Same as in \cite[Page 553]{gh14}, for any $\Omega\subset B$ and $f\in\mathcal{F}(\Omega)$,
$$\mathcal{E}(f)\ge\frac{1}{t}\left(f-P_t^Bf,f\right)\ge\frac{\|f\|_2^2}{t}\left(1-\sup\limits_{x,y\in\Omega}p_t^B(x,y)\right)\ge\frac{\|f\|_2^2}{t}\left(1-C_{\mathrm{Nash}}\frac{\mu(\Omega)}{\mu(B)}\left(\frac{W(B)}{t}\right)^{1/\nu}\right).$$
Thus by taking $t=W(B)\big(2C_{\mathrm{Nash}}\mu(\Omega)/\mu(B)\big)^\nu$ we obtain
$$\lambda_1(\Omega)=\inf\limits_{f\in\mathcal{F}(\Omega)\setminus\{0\}}\frac{\mathcal{E}(f)}{\|f\|_2^2}\ge\frac{1}{2t}=\frac{C^{-1}}{W(B)}\left(\frac{\mu(B)}{\mu(\Omega)}\right)^\nu,$$
where $C=2^{1+\nu}C_{\mathrm{Nash}}^\nu$. Hence $(\mathrm{FK})$ holds with the same $\sigma$ as in $(\mathrm{Nash}^-)$.
\end{proof}

Note that if $(\mathrm{C})$ holds and $\overline{R}<\infty$, then $(\mathrm{Nash}^-)$ can never be extended to a ball $B=B(x_0,R)$ with $R>\frac{1}{2}\overline{R}=\mathrm{diam}(M)$. Actually, here $B=M$, and thus $P_t^B=P_t$ (in particular, $P_t^B1=P_t1=1$ for all $t>0$). Replacing $\nu$ by $\nu\wedge\frac{1}{2}$ if necessary, we assume $0<\nu<1$. Hence
$$\infty=\int_{W(B)}^\infty dt=\int_{W(B)}^\infty P_t^B1dt\le\int_{W(B)}^\infty\int_B C\left(\frac{W(B)}{t}\right)^{1/\nu}dt=\frac{C\nu}{1-\nu}\mu(B)W(B)<\infty,$$
which is a contradiction. Similarly, for $(\mathrm{FK})$ and $(\mathrm{E}_\le)$, there must be $\sigma\le\frac{1}{2}$.

\begin{lemma}\label{s-}
For any $0<q\le 1$, $(\mathrm{wPH}^-_q)\Rightarrow(\mathrm{S}^-)$, where $\delta'$ could be any positive number.
\end{lemma}

\begin{proof}
Since $1\in\mathcal{F}'\subset\mathcal{F}'(B)$ and $\mathcal{E}(1,\varphi)=0$ for any non-negative $\varphi\in\mathcal{F}(B)$, then
$$u(t,x)=\begin{cases}1,&\mbox{for all }\ 0<t\le 2w'\\
P_{t-2w'}^B1(x),&\mbox{for all }\ t>2w'\end{cases}$$
is clearly caloric on $(0,\infty)\times B$ with any $w'>0$. Trivially $0\le u\le 1$ on $(0,\infty)\times M$.

Now we take $w'=\delta W(\lambda_0B)$. By $(\mathrm{wPH}^-_q)$ with $t_0=4w'$, we see
$$\mathop{\mathrm{einf}}_{(w',2w']\times\lambda_0B}P_t^B1_B=\mathop{\mathrm{einf}}_{(3w',4w']\times\lambda_0B}u\ge C_{\mathrm{wPH}}^{-1}\left(\fint_{(w',2w']\times\lambda_0B}u^qd\hat{\mu}\right)^{1/q}=C_{\mathrm{wPH}}^{-1}.$$
It follows by induction that for any integer $n\ge 1$,
$$\mathop{\mathrm{einf}}_{((2n-1)w',2nw']\times\lambda_0B}P_t^B1_B\ge C_{\mathrm{wPH}}^{-1}\left(\fint_{((2n-1)w',2nw']\times\lambda_0B}u^qd\hat{\mu}\right)^{1/q}\ge C_{\mathrm{wPH}}^{-1}\mathop{\mathrm{einf}}_{((2n-1)w',2nw']\times\lambda_0B}u\ge C_{\mathrm{wPH}}^{-n}.$$

Given any $\delta'>0$, set $n_{\delta'}=\lceil C_W\lambda_0^{-\beta_2}\frac{\delta'}{2\delta}\rceil$. For all $0<t\le\delta'W(B)$, we obtain for $\mu$-a.e.\ $x\in\lambda_0B$ that
$$P_t^B1_B(x)\ge P_t^B\left(P_{2n_{\delta'}w'-t}^B1_B\right)(x)=P_{2n_{\delta'}w'}^B1_B(x)\ge C_{\mathrm{wPH}}^{-n_{\delta'}}.$$
By standard covering technique, we obtain $(\mathrm{S}^-)$ with $\varepsilon=1-c'C_{\mathrm{wPH}}^{-n_{\delta'}}$ (where $c'$ is independent of $t,B$).
\end{proof}

\begin{proposition}\label{lle}
$(\mathrm{wPH}^-_1)\Rightarrow(\mathrm{LLE})$; while $(\mathrm{wPH}^-_q)+(\mathrm{Nash}^-)\Rightarrow(\mathrm{LLE})$  for any $0<q<1$.
\end{proposition}

Since $(\mathrm{PHI}^0)\Rightarrow(\mathrm{wPH}_1^-)$ trivially, it follows that $(\mathrm{PHI}^0)\Rightarrow(\mathrm{LLE})$, which is much simpler than \cite[Proposition 3.2]{ckw-p}.

\begin{proof}
Let $\delta$ be the constant in $(\mathrm{wPH}^-_q)$. By Lemma \ref{s-}, $(\mathrm{S}^-)$ holds with $\delta'=2\delta$. Choose a constant $\delta_L\in(0,1)$ so that for all $x_0\in M$ and $t,R>0$,
\begin{equation}\label{dltl}
W(x_0,\delta_LR)\le 4\delta W(x_0,\lambda_0R)\quad\mbox{and}\quad 4\delta_LW^{-1}(x_0,4\delta t)\le W^{-1}(x_0,t).
\end{equation}
In particular, for every ball $B=B(x_0,R)$ and all $0<t<W(\delta_LB)$,
$$B_t:=B\left(x_0,W^{-1}(x_0,t/(4\delta))\right)\subset\lambda_0B.$$

(1) For $q=1$, by $(\mathrm{wPH}^-_1)$ with $t_0=t=4\delta W(B_t)$ and $\lambda_0^{-1}B_t$ replacing $B$, we see
\begin{align*}
\int_Dp_t^{\lambda_0^{-1}B_t}(x,dy)\ &=P_t^{\lambda_0^{-1}B_t}1_D(x)\ge\fint_{\left(\frac{t}{4},\frac{t}{2}\right]\times B_t}P_s^{\lambda_0^{-1}B_t}1_D(z)d\hat{\mu}(s,z)=\fint_{t/4}^{t/2}\frac{1}{\mu(B_t)}\int_{B_t}\int_Dp_s^{\lambda_0^{-1}B_t}(z,dy)d\mu(z)ds\\
&=\fint_{t/4}^{t/2}\frac{1}{\mu(B_t)}\int_D\int_{B_t}p_s^{\lambda_0^{-1}B_t}(y,dz)d\mu(y)ds=\fint_{t/4}^{t/2}\frac{1}{\mu(B_t)}\int_DP_s^{\lambda_0^{-1}B_t}1_{B_t}d\mu ds\\
&\ge\fint_{t/4}^{t/2}\frac{1}{\mu(B_t)}\int_DP_s^{B_t}1_{B_t}d\mu ds\ge\frac{1-\varepsilon}{\mu(B_t)}\mu(D)
\end{align*}
for any $\mu$-measurable $D\subset B_t$ and $\mu$-a.e.\ $x\in\frac{1}{4}B_t$. Here we use $(\mathrm{S}^-)$ in the last line. Hence
$$p_t^B(x,dy)\ge p_t^{\lambda_0^{-1}B_t}(x,dy)\ge\frac{1-\varepsilon}{\mu(B_t)}\quad\mbox{on}\quad\frac{1}{4}B_t\supset B(x_0,\delta_LW^{-1}(x_0,t)).$$
Then $(\mathrm{LLE})$ follows by Lemma \ref{LE1}.

(2) For $0<q<1$, if $0<R<\sigma\overline{R}$, where $\sigma$ comes from $(\mathrm{Nash}^-)$, then
\begin{align*}
\int_{B_t}\left(p_s^{\lambda_0^{-1}B_t}(x,y)\right)^qd\mu(y)\ge&\ \left(\int_{B_t}p_s^{\lambda_0^{-1}B_t}(x,y)d\mu(y)\right)^{2-q}\left(\int_{B_t}\left(p_s^{\lambda_0^{-1}B_t}(x,y)\right)^2d\mu(y)\right)^{q-1}\\
\ge&\ \left(P_s^{B_t}1_{B_t}(x)\right)^{2-q}\left(p_{2s}^{\lambda_0^{-1}B_t}(x,x)\right)^{-(1-q)}\ge(1-\varepsilon)^{2-q}\left(\frac{\mu(B_t)}{C}\left(\frac{2s}{W(B_t)}\right)^{1/\nu}\right)^{1-q}\ge\varepsilon'\mu(B_t)^{1-q}
\end{align*}
for all $\frac{t}{4}<s\le\frac{t}{2}$ and $x\in\frac{1}{4}B_t$ by H\"older inequality, $(\mathrm{S}^-)$ and $(\mathrm{Nash}^-)$ successively. Same as in (1), it follows by $(\mathrm{wPH}^-_q)$ that for $\mu$-a.e.\ $x,y\in\frac{1}{4}B_t$,
$$p_t^B(x,y)\ge p_t^{\lambda_0^{-1}B_t}(x,y)\ge C^{-1}\left(\fint_{\left(\frac{t}{4},\frac{t}{2}\right]\times B_t}\left(p_s^{\lambda_0^{-1}B_t}(x,z)\right)^qd\hat{\mu}(s,z)\right)^{1/q}\ge C^{-1}\left(\frac{\varepsilon'\mu(B_t)^{1-q}}{\mu(B_t)}\right)^{1/q}\ge\frac{C^{-2}(\varepsilon')^{1/q}}{V\left(x_0,W^{-1}(x_0,t)\right)},$$
showing $(\mathrm{LLE}^-)$ again. Hence $(\mathrm{LLE})$ follows by Lemma \ref{LE1}.
\end{proof}

Now we show the necessity of the basic assumptions $(\mathrm{VD})$ and $(\mathrm{RVD})$.

\begin{lemma}
$(\mathrm{S}^-)+(\mathrm{Nash^-})\Rightarrow(\mathrm{VD})$.
\end{lemma}

\begin{proof}
For any $B=B(x,r)$ with $r<r_0:=\frac{1}{2}\sigma\overline{R}$ (where $\sigma$ comes from $(\mathrm{Nash^-})$), we obtain by taking $t=\delta'W(B)$ (where $\delta'$ comes from $(\mathrm{S}^-)$) that
$$1-\varepsilon\le P_t^B1_B(x)\le P_t^{2B}1_B(x)\le\mu(B)\mathop{\mathrm{esup}}_{y\in B}p_t^{2B}(x,y)\le\frac{C_N\mu(B)}{\mu(2B)}\left(\frac{W(2B)}{t}\right)^{1/\nu}\le C_{N,W}\frac{\mu(B)}{\mu(2B)}.$$
This is exactly the condition $(\mathrm{VD})$ if $\overline{R}=\infty$; while if $\overline{R}<\infty$, then $M$ is compact since it is the closure of every $B(x,\overline{R})$. Therefore, $M$ could be covered by a finite number of balls $B(x_i,\frac{1}{2}r_0)$, $i=1,\cdots,N$. In particular, for any $B=B(x,r)$ with $r>r_0$, there exists $1\le i\le N$ such that $x\in B(x_i,\frac{1}{2}r_0)$. Therefore, $B(x_i,\frac{1}{2}r_0)\subset B(x,r)$, implying that
$$\frac{\mu(2B)}{\mu(B)}\le\frac{\mu(M)}{V(x_i,\frac{1}{2}r_0)}\le\frac{\mu(M)}{\inf\limits_{1\le i'\le N}V(x_{i'},\frac{1}{2}r_0)}.$$
The right side is a constant decided by the space $(M,d,\mu)$, which is finite since $\mu$ is a Radon measure of full support. Thus $(\mathrm{VD})$ is proved.
\end{proof}

\begin{lemma}
$(\mathrm{VD})+(\mathrm{FK})+(\mathrm{TJ})\Rightarrow(\mathrm{RVD})$.
\end{lemma}

\begin{proof}
Same as in \cite[Lemma 6.9]{hl}.
\end{proof}

\section{Proof of Theorem \ref{M2}}\label{PM2}

\begin{lemma}\label{L1-E}
$(\mathrm{S}^-)+(\mathrm{FK})\Rightarrow(\mathrm{E})$. In particular, $(\mathrm{LLE})\Rightarrow(\mathrm{E})$.
\end{lemma}

\begin{proof}
Since $(\mathrm{LLE})\Rightarrow(\mathrm{S}^-)+(\mathrm{FK})$ by (\ref{vLE}), it suffices to prove the first assertion.

Recall that $(\mathrm{FK})\Rightarrow(\mathrm{E}_\le)$ by \cite[Lemma 12.2]{ghh+1}. On the other hand, by $(\mathrm{S}^-)$, we have $P_t^B1_B\ge\varepsilon$ on $\frac{1}{4}B$ for every ball $B=B(x_0,R)$ and all $0<t<\delta'W(B)$. It follows that
$$G^B1_B=\int_0^\infty P_t^B1_Bdt\ge\int_0^{\delta'W(B)}P_t^B1_Bdt\ge\varepsilon\delta'W(B)$$
on $\frac{1}{4}B$, that is, $(\mathrm{S}^-)\Rightarrow(\mathrm{E}_\ge)$. The proof is then completed.
\end{proof}

\begin{proposition}\label{v_D}
For any $0<q\le 1$, $(\mathrm{PGL}_q)+(\mathrm{TJ})\Rightarrow(\mathrm{PHR})$.
\end{proposition}

\begin{proof}
It suffices to consider the case where $u$ is non-negative on $Q$ and $\|u\|_Q+W(B)^{1-\frac{1}{p}}T_Q^{(\frac{2}{3};p)}(u)=1$.

Let $\Lambda\ge 2\vee\lambda_0^{-1}\vee(4C_W)^{1/\beta_1}$ be a constant to be determined later. Set $B_0=B$ and $Q_0=Q$. For all $x,y\in\Lambda^{-1}B$ and $t,s\in(t_0-\delta W(\Lambda^{-1}B),t_0]$, say $t\ge s$, let
$$B_k=B(x,\Lambda^{-k}R)\quad\mbox{and}\quad Q_k:=(t-\delta W(B_k),t]\times B_k.$$
Then
$$(s,y)\in Q_k\quad\mbox{with}\quad k=\left[\log_\Lambda\left(\frac{R}{d(x,y)}\wedge\frac{R}{W^{-1}(x,(t-s)/\delta)}\right)\right].$$

It suffices to prove there exits $\frac{1}{2}<\eta<1$ such that for all $k\ge 0$,
\begin{equation}\label{indk}
b_k-a_k\le\eta^k,\quad\mbox{where}\quad a_k=\mathop{\mathrm{einf}}\limits_{(s,y)\in Q_k}u(s,y)\quad\mbox{and}\quad b_k=\mathop{\mathrm{esup}}\limits_{(s,y)\in Q_k}u(s,y),
\end{equation}
and then $(\mathrm{PHR}^+)$ follows by observing
$$|u(t,x)-u(s,y)|\le\eta^k\le\eta^{\log_\Lambda\left(\frac{R}{d(x,y)}\wedge\frac{R}{W^{-1}(x,(t-s)/\delta)}\right)-1}\le C(\eta,\delta,\Lambda)\left(\frac{d(x,y)}{R}\vee\frac{W^{-1}(x,t-s)}{R}\right)^{\log_\Lambda\eta^{-1}}.$$

Actually, we prove (\ref{indk}) by induction. It is trivial that $b_0-a_0\le 1$. Now assume $b_i-a_i\le\eta^{i-1}$ for all $i=1,\dotsc,k$. Recall by the range of $\Lambda$ that
$$Q'_{k+1}:=\big(t-4\delta W(B_{k+1}),t\big]\times\lambda_0^{-1}B_{k+1}\subset Q_k.$$
In particular, $a_k\le u\le b_k$ on $Q'_{k+1}$. Denote also
$$Q^-_{k+1}:=\big(t-3\delta W(B_{k+1}),t-2\delta W(B_{k+1})\big]\times B_{k+1}.$$

By considering $(\|u\|_Q-u)$ instead of $u$ if necessary, we assume $$\hat{\mu}(D_{k+1})\ge\frac{1}{2}\hat{\mu}\left(Q^-_{k+1}\right),\quad\mbox{where}\quad D_{k+1}:=\left\{(s,y)\in Q^-_{k+1}:u(s,y)\le\frac{a_k+b_k}{2}\right\}.$$
Applying $(\mathrm{PGL}_q)$ to $(b_k-u)$ on $Q'_{k+1}$ with $\lambda=\lambda_0$, we see on $Q_{k+1}$ that
\begin{align*}
b_k-u\ge&\ C_q^{-1}\left(\frac{1}{2}\right)^{1/q}\left(b_k-\frac{b_k+a_k}{2}\right)-2\int_{t-4\delta W(B_{k+1})}^tT_{\lambda_0^{-1}B_{k+1}}^{(2/3)}((b_k-u)_-;s')ds'\\
=&\ \frac{b_k-a_k}{2^{1+1/q}C_q}-2\int_{t-4\delta W(B_{k+1})}^t\mathop{\mathrm{esup}}\limits_{\frac{2}{3}\lambda_0^{-1}B_{k+1}}\left\{\int_{B^c}+\int_{B\setminus B_k}\right\}(b_k-u)_-(s',w)J(\cdot,dw)\\
\ge&\ \frac{b_k-a_k}{2^{1+1/q}C_q}-2\left(4\delta W(B_{k+1})\right)^{1-\frac{1}{p}}T_Q^{(\frac{2}{3};p)}(u)-8\delta W(B_{k+1})\mathop{\mathrm{esup}}\limits_{\frac{2}{3}\lambda_0^{-1}B_{k+1}}\sum\limits_{i=0}^{k-1}\int_{B_i\setminus B_{i+1}}(b_k-u)_-(s',w)J(\cdot,dw).
\end{align*}
Provided that $\Lambda^{\beta_1(1-\frac{1}{p})}>\eta^{-1}$, it follows by $(\mathrm{TJ})$ and the inductive assumption that
\begin{align*}
b_{k+1}-a_{k+1}=&\ \mathop{\mathrm{esup}}\limits_{Q_{k+1}}u-a_{k+1}\le b_k-\mathop{\mathrm{einf}}\limits_{Q_{k+1}}(b_k-u)-a_k\\
\le&\ b_k-a_k-\frac{b_k-a_k}{2^{1+1/q}C_q}+2\left(4\delta W(B_{k+1})\right)^{1-\frac{1}{p}}T_Q^{(\frac{2}{3};p)}(u)+8\delta W(B_{k+1})\mathop{\mathrm{esup}}\limits_{\frac{2}{3}\lambda_0^{-1}B_{k+1}}\sum\limits_{i=0}^{k-1}\int_{B_i\setminus B_{i+1}}(b_i-a_i)J(\cdot,dw)\\
\le&\ (1-\varepsilon_q)(b_k-a_k)+2\left(\frac{4\delta W(B_{k+1})}{W(B)}\right)^{1-\frac{1}{p}}+8\delta W(B_{k+1})\sum\limits_{i=0}^{k-1}(b_i-a_i)\frac{C_{3/8}}{W(B_{i+1})}\\
\le&\ (1-\varepsilon_q)\eta^k+C_{\delta,p}\Lambda^{-(k+1)(1-\frac{1}{p})\beta_1}+C_\delta\sum\limits_{i=0}^{k-1}\eta^i\Lambda^{-(k-i)\beta_1}\le(1-\varepsilon_q)\eta^k+C_{\delta,p}\Lambda^{-(1-\frac{1}{p})\beta_1}\eta^k+\frac{C_\delta}{\eta\Lambda^{\beta_1}-1}\eta^k,
\end{align*}
where $\varepsilon_q=2^{-1-1/q}C_q^{-1}$. Therefore, we obtain $b_{k+1}-a_{k+1}\le\eta^{k+1}$ by letting
$$\eta=1-\frac{\varepsilon_q}{3}\quad\mbox{and}\quad\Lambda\ge\left(3C_{\delta,p}\varepsilon_q^{-1}\right)^{\beta_1^{-1}/(1-\frac{1}{p})}\vee\left(\eta^{-1}\left(1+3C_\delta\varepsilon_q^{-1}\right)\right)^{\beta_1^{-1}}.$$
The proof is then completed.
\end{proof}

\begin{lemma}\label{in-hr}
$(\mathrm{PHR})+(\mathrm{TJ})\Rightarrow(\mathrm{PHR}^0)$ and $(\mathrm{EHR})+(\mathrm{TJ})\Rightarrow(\mathrm{EHR}^0)$.
\end{lemma}

\begin{proof}
With $Q=(t_0-W(B),t_0]\times B$, clearly $\|u\|_Q\le\|u\|_{(t_0-W(B),t_0]\times M}$. Further, for any $p\in(1,\infty]$, 
\begin{align*}
W(B)^{1-\frac{1}{p}}T_Q^{(\lambda;p)}(u)=&\ W(B)^{1-\frac{1}{p}}\left(\int_{t_0-W(B)}^{t_0}\left(\mathop{\mathrm{esup}}\limits_{x\in\lambda B}\int_{B^c}|u(y)|J(x,dy)\right)^pds\right)^{1/p}\\
\le&\ W(B)^{1-\frac{1}{p}}\|u\|_{(t_0-W(B),t_0]\times M}\left(\int_{t_0-W(B)}^{t_0}\left(\mathop{\mathrm{esup}}\limits_{x\in\lambda B}\int_{B^c}J(x,dy)\right)^pds\right)^{1/p}\\
\le&\ W(B)^{1-\frac{1}{p}}\|u\|_{(t_0-W(B),t_0]\times M}\left(W(B)\left(\frac{C_\lambda}{W(B)}\right)^p\right)^{1/p}=C_\lambda\|u\|_{(t_0-W(B),t_0]\times M},
\end{align*}
where $C_\lambda$ comes from $(\mathrm{TJ})$, showing $(\mathrm{PHR})+(\mathrm{TJ})\Rightarrow(\mathrm{PHR}^0)$ clearly.

The elliptic analog follows with $p=\infty$ and $u(t,x)=f(x)$.
\end{proof}

\begin{proposition}\label{e-lle}
$(\mathrm{EHR})+(\mathrm{E})\Rightarrow(\mathrm{LLE})$.
\end{proposition}

\begin{proof}
The proof inherits \cite[Subsection 4.2]{ckw-p} with slight modification.

(1) We show $(\mathrm{EHR})+(\mathrm{E}_\le)\Rightarrow(\mathrm{FK})$, so that $(\mathrm{Nash}^-)$ follows by Lemma \ref{nash}.

Fix $B=B(x_0,R)$ with $0<R<\frac{\sigma}{2}\overline{R}$, where $\sigma$ comes from $(\mathrm{E}_\le)$. For every open $U\subset B$, with any $y\in U$, set
$$B_y:=B(y,2R)\supset B\quad\mbox{and}\quad h_U:=G^{B_y}1_U-G^{\varepsilon B_y}1_U.$$
Clearly $h_U$ is harmonic in $\varepsilon B$ for every $\varepsilon\in(0,\delta)$, where $\delta$ comes from $(\mathrm{EHR})$. Thus
\begin{align*}
\mathop{\mathrm{esup}}\limits_{\varepsilon^2B_y}G^{B_y}1_U&\ \le\mathop{\mathrm{esup}}\limits_{\varepsilon^2B_y}h_U+\mathop{\mathrm{esup}}\limits_{\varepsilon^2B_y}G^{\varepsilon B_y}1_U\le\fint_{\varepsilon^2B_y}h_Ud\mu+\mathop{\mathrm{eosc}}\limits_{2B_y}h_U+\mathop{\mathrm{esup}}\limits_{\varepsilon^2B_y}G^{\varepsilon B_y}1_U\\
&\ \le\frac{1}{\mu(\varepsilon^2B_y)}\int_UG^{B_y}1d\mu+C\varepsilon^\gamma\mathop{\mathrm{esup}}\limits_{B_y}G^{B_y}1+\mathop{\mathrm{esup}}\limits_{\varepsilon^2B_y}G^{\varepsilon B_y}1\\
&\ \le\left(\frac{\mu(U)}{\mu(\varepsilon^2B_y)}+C\varepsilon^\gamma\right)C_{\mathrm{E}}W(B_y)+C_{\mathrm{E}}W(\varepsilon B_y)\le C'W(B)\left(\frac{\mu(U)}{\mu(B)}\varepsilon^{-2\alpha}+\varepsilon^{\gamma\wedge\beta_1}\right),
\end{align*}
where we use $(\mathrm{EHR})$ and $0\le h_U\le G^{B_y}1_U$ on the second line, and use $(\mathrm{E}_\le)$ on the third.

If $\frac{\mu(U)}{\mu(B)}\le\delta^{2\alpha+\gamma\wedge\beta_1}$, then by optimizing the right side over $\varepsilon$, we obtain
$$\lambda_1(U)^{-1}\le\mathop{\mathrm{esup}}\limits_{U}G^U1_U\le\sup\limits_{y\in U}\mathop{\mathrm{esup}}\limits_{\varepsilon^2B_y}G^{B_y}1_U\le C'W(B)\left(\frac{\mu(U)}{\mu(B)}\right)^{1/(1+\frac{2\alpha}{\gamma\wedge\beta_1})}.$$
The same is trivially true if $\frac{\mu(U)}{\mu(B)}>\delta^{2\alpha+\gamma\wedge\beta_1}$. Thus $(\mathrm{FK})$ is proved with $\nu=\big(1+\frac{2\alpha}{\gamma\wedge\beta_1}\big)^{-1}$.

(2) For $t<W(B)$, set $B'=B\big(x_0,W^{-1}(x_0,t)\big)$. For all $x\in\frac{1}{2}B'$ and $0<r'<\frac{1}{2}W^{-1}(x_0,t)$, we prove
\begin{equation}\label{dle}
\mathop{\mathrm{eosc}}\limits_{B(x,r')}p_t^{B'}(x,\cdot)\le\frac{C''}{V\left(x_0,W^{-1}(x_0,t)\right)}\left(\frac{r'}{W^{-1}(x_0,t)}\right)^{\frac{\beta_1\gamma}{\beta_1+\gamma}}.
\end{equation}

Indeed, we clearly have $B(x,r)\subset B'$ for any $r'<r<\frac{1}{2}W^{-1}(x_0,t)$. Further,
$$p_t^{B'}(x,\cdot)=-G^{B'}f_t\quad\mbox{and}\quad\|f_{2t}\|_{L^\infty}\le\frac{1}{t}\mathop{\mathrm{esup}}\limits_{y\in B'}p_t^{B'}(y,y),\quad\mbox{where}\quad f_t=\partial_tp_t^{B'}(x,\cdot)$$
by \cite[Proof of Lemma 4.8]{ckw-p}. Therefore, by $(\mathrm{EHR})$, $(\mathrm{E}_\le)$ and $(\mathrm{Nash}^-)$,
\begin{align*}
\mathop{\mathrm{eosc}}\limits_{B(x,r')}p_t^{B'}(x,\cdot)\le&\ \mathop{\mathrm{eosc}}\limits_{B(x,r')}\left(G^{B'}f_t-G^{B(x,r)}f_t\right)+\mathop{\mathrm{eosc}}\limits_{B(x,r')}G^{B(x,r)}f_t\\
\le&\ C\left(\frac{r'}{r}\right)^\gamma\left\|G^{B'}f_t-G^{B(x,r)}f_t\right\|_\infty+2\left\|G^{B(x,r)}1\right\|_\infty\|f_t\|_{L^\infty(B(x,r))}\\
\le&\ 2C\left(\frac{r'}{r}\right)^\gamma\mathop{\mathrm{esup}}\limits_{x',y'\in B'}p_t^{B'}(x',y')+2C\frac{1}{t}W(x,r)\mathop{\mathrm{esup}}\limits_{y\in B'}p_{\frac{t}{2}}^{B'}(y,y)\\
\le&\ \frac{C'}{\mu(B')}\left\{\left(\frac{r'}{r}\right)^\gamma+\left(\frac{r}{W^{-1}(x,t)}\right)^{\beta_1}\right\}\le\frac{C''}{V\left(x_0,W^{-1}(x_0,t)\right)}\left\{\left(\frac{r'}{r}\right)^\gamma+\left(\frac{r}{W^{-1}(x_0,t)}\right)^{\beta_1}\right\}.
\end{align*}
Then (\ref{dle}) follows by optimizing the right side over $r\in\left(r',W^{-1}(x_0,t)\right)$.

(3) Note that $(\mathrm{S}^-)$ holds by (\ref{13-4}). Thus with $B'=B\left(x,W^{-1}(x,t/\delta')\right)\subset B$,
$$p_t^B(x,x)\ge p_t^{B'}(x,x)=\int_{B'}\left(p_{\frac{t}{2}}^{B'}(x,y)\right)^2d\mu(y)\ge\frac{1}{\mu(B')}\left(\int_{B'}p_{\frac{t}{2}}^{B'}(x,y)d\mu(y)\right)^2\ge\frac{(1-\varepsilon)^2}{V\left(x,W^{-1}(x,t)\right)}.$$
$(\mathrm{LLE})$ follows easily by combining this inequality with (\ref{dle}).
\end{proof}

\section{Caloric extension and tails in Harnack inequalities}\label{Tail}

\begin{proposition}\label{wph+t}
$(\mathrm{wPH}_q^0)\Leftrightarrow(\mathrm{wPH}_q)$ for every $0<q<1$.
\end{proposition}

\begin{proof}
It suffices to prove ``$\Rightarrow$'', since ``$\Leftarrow$'' is trivial. Fix $0<\lambda\le\frac{2}{3}\lambda_0$, and set $w':=W(\lambda B)$.

For every bounded, non-negative and supercaloric $u$ on $Q=(t_0-4\delta w',t_0]\times B$, if $\overline{T}_Q^{(2/3)}(u_-)=\infty$, then $(\mathrm{wPH}_q)$ holds trivially for $u$; while if $\overline{T}_Q^{(2/3)}(u_-)<\infty$, then by the proof of Lemma \ref{comp}, $\bar{u}:=u_{t_0-4\delta w'}^{\lambda_0^{-1}\lambda B;B}$ is bounded and supercaloric on $(t_0-4\delta w',t_0]\times\frac{\lambda}{\lambda_0}B$, globally non-negative. Applying $(\mathrm{wPH}_q^0)$ to $\bar{u}$, we obtain
$$\left(\fint_{Q_-}u^qd\hat{\mu}\right)^{1/q}\le\left(\fint_{Q_-}\left(\bar{u}\right)^qd\hat{\mu}\right)^{1/q}\le C\mathop{\mathrm{einf}}\limits_{Q_+}\bar{u}=C\mathop{\mathrm{einf}}\limits_{Q_+}\left(u+2\int_{t_0-4\delta w'}^tT_B^{(\lambda/\lambda_0)}(u_-,s)ds\right)\le C\mathop{\mathrm{einf}}\limits_{Q_+}u+2C\overline{T}_Q^{(2/3)}(u_-),$$
where $Q_-=\left(t_0-3\delta w',t_0-2\delta w'\right]\times\lambda B$ and $Q_+=\left(t_0-\delta w',t_0\right]\times\lambda B$. That is, $(\mathrm{wPH}_q)$ holds with $\delta$ unchanged, $\lambda_0$ replaced by $\frac{2}{3}\lambda_0$, and $C$ replaced by $2C$.
\end{proof}

Given every two open sets $U,\Omega$ in $M$ with $U\Subset\Omega$, for every $\hat{\mu}$-measurable $f$ such that $\overline{T}_{(t_1,t_2]\times\Omega}^U(f)<\infty$ and $\mathrm{supp}(f(t,\cdot))\subset\Omega^c$ for all $t_1<t\le t_2$, define the \emph{caloric extension} of $f$ into $(t_1,t_2]\times U$ as
\begin{equation}\label{ext-d}
f_{t_1,U}(t,x)=\begin{cases}f(t,x),&\mbox{if }x\in U^c;\\
2\int_{t_1}^t\int_U\int_{\Omega^c}f(s,y)J(z,dy)p_{t-s}^U(x,dz)ds,&\mbox{if }x\in U\end{cases}
\end{equation}
on $(t_1,t_2]\times M$. Note that $f_{t_1,U}\equiv 0$ in $U$ if $(\mathcal{E},\mathcal{F})$ is strongly local.

\begin{lemma}\label{ext-F}
Under the assumptions above,

(i) $f_{t_1,U}$ is bounded on $(t_1,t_2]\times U$. To be exact, for $\mu$-a.e.\ $x\in\Omega$,
\begin{equation}\label{ext-infty}
\left|f_{t_1,U}(t,x)\right|\le 2\overline{T}_{(t_1,t]\times\Omega}^U(f);
\end{equation}

(ii) if $T_{(t_1,t_2]\times\Omega}^{U;2}(f)<\infty$, then $1_\Omega f_{t_1,U}(t,\cdot)\in\mathcal{F}(U)$ for all $t_1<t\le t_2$;

(iii) if further $f(t,\cdot)\in\mathcal{F}'(U)$ for all $t_1<t\le t_2$, then $f_{t_1,U}\in\mathcal{F}'(U)$, and is caloric on $(t_1,t_2]\times U$.
\end{lemma}

Note in particular that if $f(t,\cdot)\in\mathcal{F}'$ for all $t$, then $f_{t_1,U}\in\mathcal{F}'$; while if $f(t,\cdot)\in\mathcal{F}$ for all $t$, then $f_{t_1,U}\in\mathcal{F}$.

\begin{proof}
For simplicity we set
$$T^f_s(z)=\int_{\Omega^c}f(s,y)J(z,dy).$$

(i) By direct computation:
$$\left|f_{t_1,U}(t,x)\right|=2\left|\int_{t_1}^t\int_UT^f_s(z)p_{t-s}^U(x,dz)ds\right|\le 2\int_{t_1}^t\left(\mathop{\mathrm{esup}}\limits_UT^{|f|}_s\right)\int_Up_{t-s}^U(x,dz)ds\le 2\overline{T}_{(t_1,t]\times\Omega}^U(f).$$

(ii) Since $U\Subset\Omega$, we see by (i) that $f_{t_1,U}\in L^\infty(U)\subset L^2(U)$. Recall that for all $\sigma>0$, the bilinear form
$$\mathcal{E}_\sigma^U:(\varphi,\phi)\mapsto\sigma^{-1}(\varphi-P_\sigma^U\varphi,\phi)$$
is symmetric and non-negative definite on $L^2(U)$. Therefore,
\begin{align*}
\mathcal{E}_\sigma^U(1_Uf_{t_1,U})=&\ 4\int_{t_1}^t\int_{t_1}^t\mathcal{E}_\sigma^U\bigg(P_{t-s}^UT^f_s,P_{t-s'}^UT^f_{s'}\bigg)dsds'=4\int_{t_1}^t\int_{t_1}^t\mathcal{E}_\sigma^U\bigg(P_{t-s'}^UT^f_s,P_{t-s}^UT^f_{s'}\bigg)dsds'\\
\le&\ 2\int_{t_1}^t\int_{t_1}^t\left\{\mathcal{E}_\sigma^U\bigg(P_{t-s'}^UT^f_s,P_{t-s'}^UT^f_s\bigg)+\mathcal{E}_\sigma^U\bigg(P_{t-s}^UT^f_{s'},P_{t-s}^UT^f_{s'}\bigg)\right\}dsds'\\
\le&\ 4\int_{t_1}^t\int_{t_1}^t\mathcal{E}\bigg(P_{t-s'}^UT^f_s,P_{t-s'}^UT^f_s\bigg)dsds'=4\int_{t_1}^t\int_{t_1}^t\left(-\frac{1}{2}\frac{\mathrm{d}}{\mathrm{d}\tau}\bigg|_{\tau=t-s'}\bigg(P_\tau^UT^f_s,P_\tau^UT^f_s\bigg)_{L^2(U)}\right)ds'ds\\
\le&\ 2\int_{t_1}^t\bigg(T^f_s,T^f_s\bigg)_{L^2(U)}ds\le 2\int_{t_1}^t\left(\mathop{\mathrm{esup}}\limits_{z\in U}T^{|f|}_s\right)^2\mu(U)ds\le 2\mu(U)T_{(t_1,t]\times\Omega}^{U;2}(f)<\infty
\end{align*}
uniformly on $\sigma>0$. Thus by \cite[Lemma 1.3.4]{fot},
$$\mathcal{E}_\sigma^U(1_Uf_{t_1,U})\uparrow\mathcal{E}(1_Uf_{t_1,U})<\infty\quad\mbox{as}\quad\sigma\downarrow 0,$$
showing that $1_Uf_{t_1,U}\in\mathcal{F}(U)$.

(iii) Observe that $f_{t_1,U}=f+1_Uf_{t_1,U}$. Hence $f_{t_1,U}\in\mathcal{F}'(U)$ by (ii). Further,
\begin{align*}
\frac{\mathrm{d}}{\mathrm{d}t}\left(f_{t_1,U}(t,\cdot),\varphi\right)=&\ 2\frac{\mathrm{d}}{\mathrm{d}t}\int_{t_1}^t\left(P_{t-s}^UT^f_s,\varphi\right)ds=2(T^f_s,\varphi)+2\int_{t_1}^t\left(\varphi,\partial_tP_{t-s}^UT^f_s\right)ds\\
=&\ 2\int_U\int_{\Omega^c}\varphi(x)f(t,y)J(x,dy)d\mu(x)-2\int_{t_1}^t\mathcal{E}\left(P_{t-s}^UT^f_s,\varphi\right)ds\\
=&\ -\mathcal{E}(f(t,\cdot),\varphi)-\mathcal{E}(1_Uf_{t_1,U},\varphi)=-\tilde{\mathcal{E}}(f_{t_1,U},\varphi)
\end{align*}
for all $\varphi\in\mathcal{F}(U)$, that is, $f_{t_1,U}$ is caloric.
\end{proof}

\begin{proposition}\label{phi+t}
$(\mathrm{PHI}^0)\Leftrightarrow(\mathrm{PHI})$.
\end{proposition}

\begin{proof}
Similar as in Proposition \ref{wph+t}, it suffices to prove ``$\Rightarrow$''.

For any $u$ that is non-negative bounded caloric on $Q$, if $\overline{T}_Q^{(3/4)}(u_-)=\infty$, then $(\mathrm{PHI})$ trivially holds for $u$. Thus we only consider the case $\overline{T}_Q^{(3/4)}(u_-)<\infty$.

Denote $w=\delta W(B)$ and $w'=\delta W(\frac{3}{4}B)$. For any $\varepsilon\in(0,w-w')$, define
\begin{equation}\label{uep}
u_\varepsilon(t,x):=\frac{1}{\varepsilon}\int_0^\varepsilon u(t-s',x)ds'\quad\mbox{and}\quad \hat{u}_\varepsilon:=u_\varepsilon+((u_\varepsilon)_-)_{t_0-4w+\varepsilon,\frac{3}{4}B}.
\end{equation}
Clearly $u_\varepsilon$ is caloric and bounded on $Q_\varepsilon=(t_0-4w+\varepsilon,t_0]\times B$. Further,
$$T_{(t_0-4w+\varepsilon,t_0]\times B}^{(3/4)}((u_\varepsilon)_-)\le\sup\limits_{t_0-4w+\varepsilon<t\le t_0}\frac{1}{\varepsilon}\int_{t-\varepsilon}^t T_B^{(3/4)}(u_-;s)ds\le\frac{1}{\varepsilon}\overline{T}_Q^{(3/4)}(u_-)<\infty.$$
Thus by Lemma \ref{ext-F}, $\hat{u}_\varepsilon$ is caloric and bounded on $(t_0-4w+\varepsilon,t_0]\times\frac{3}{4}B$. It is easy to check that $\hat{u}_\varepsilon\ge 0$ globally with the help of Proposition \ref{gpmp}.

For any $t_0-3w<t\le t_0-2w$ and any non-negative integer $n$, set
$$\tau_n=t+2w'+nt'\quad\mbox{with some}\quad t'\in\big(w',w\wedge 2w'\wedge(t+4w-t_0)\big).$$
Let $n_\ast$ be the smallest integer such that $\tau_{n_\ast}>t_0-w$. Then $n_\ast<2C_W\left(\frac{3}{2}\right)^{\beta_2}=:C'_W$ by (\ref{W}). Further,
$$t-w'>t-w+\varepsilon>t_0-4w+\varepsilon\quad\mbox{and}\quad\tau_0=t+2w'<t+2w\le t_0.$$
It follows that $t_0-t'<\tau_{n_\ast}<t_0$. Applying $(\mathrm{PHI}^0)$ to $\hat{u}_\varepsilon$ in $(t-t',t-t'+4w']\times\frac{3}{4}B$ and $(\tau_n-2w',\tau_n+2w']\times\frac{3}{4}B$ successively with $0\le n<n_\ast$, we obtain
$$\mathop{\mathrm{esup}}_{\frac{3}{4}\lambda_0B}\hat{u}_\varepsilon(t,\cdot)\le C\mathop{\mathrm{einf}}_{\frac{3}{4}\lambda_0B}\hat{u}_\varepsilon(\tau_0,\cdot)\le C^2\mathop{\mathrm{einf}}_{\frac{3}{4}\lambda_0B}\hat{u}_\varepsilon(\tau_1,\cdot)\le\cdots\le C^{1+n_\ast}\mathop{\mathrm{einf}}_{\frac{3}{4}\lambda_0B}\hat{u}_\varepsilon(\tau_{n_\ast},\cdot).$$
Therefore, with $Q_-:=(t_0-3w,t_0-2w]\times\frac{3}{4}\lambda_0B$ and $Q_+:=(t_0-w,t_0]\times\frac{3}{4}\lambda_0B$, there is
\begin{align*}
\mathop{\mathrm{esup}}_{Q_-}u_\varepsilon\le&\ \mathop{\mathrm{esup}}_{(t_0-3w,t_0-2w]\times\frac{3}{4}\lambda_0B}\hat{u}_\varepsilon=\sup\limits_{t_0-3w<t\le t_0-2w}\mathop{\mathrm{esup}}_{\frac{3}{4}\lambda_0B}\hat{u}_\varepsilon(t,\cdot)\le C^{1+n_\ast}\mathop{\mathrm{einf}}_{\frac{3}{4}\lambda_0B}\hat{u}_\varepsilon(\tau_{n_\ast},\cdot)\le C^{C'_W}\mathop{\mathrm{einf}}\limits_{Q_+}\hat{u}_\varepsilon\\
=&\ C^{C'_W}\mathop{\mathrm{einf}}\limits_{Q_+}\left(u_\varepsilon+2\int_{t_0-4w+\varepsilon}^t\int_{\frac{3}{4}B}\int_{B^c}(u_\varepsilon)_-(s,y)J(z,dy)p_{t-s}^{\frac{3}{4}B}(x,dz)ds\right)\\
\le&\ C^{C'_W}\mathop{\mathrm{einf}}\limits_{Q_+}u_\varepsilon+2C^{C'_W}\mathop{\mathrm{esup}}\limits_{(t,x)\in Q_+}\int_{t_0-4w+\varepsilon}^t\left\{\frac{1}{\varepsilon}\int_{s-\varepsilon}^sT_B^{(3/4)}(u_-;s')ds'\int_{\frac{3}{4}B}p_{t-s}^{\frac{3}{4}B}(x,dz)\right\}ds\\
\le&\ C^{C'_W}\mathop{\mathrm{einf}}\limits_{Q_+}u_\varepsilon+8C^{C'_W}\overline{T}_Q^{(3/4)}(u_-).
\end{align*}
Letting $\varepsilon\downarrow 0$, by the weak $L^2(B)$-continuity of $u$, it follows that
$$\mathop{\mathrm{esup}}_{Q_-}u\le C^{C'_W}\mathop{\mathrm{einf}}\limits_{Q_+}u+8C^{C'_W}\overline{T}_Q^{(3/4)}(u_-),$$
that is, $(\mathrm{PHI})$ holds with $\delta$ unchanged, $\lambda_0$ replaced by $\frac{3}{4}\lambda_0$, and $C$ replaced by $8C^{C'_W}$.
\end{proof}

A similar strategy works for the elliptic case with a further assumption $(\mathrm{E}_\le)$:

\begin{proposition}\label{ehi+t}
If $(\mathrm{E}_\le)$ holds, then $(\mathrm{wEH}_q^0)\Leftrightarrow(\mathrm{wEH}_q)$ for all $q\in(0,1]$, and $(\mathrm{EHI}^0)\Leftrightarrow(\mathrm{EHI})$.
\end{proposition}

\begin{proof}
It suffices to prove ``$\Rightarrow$''. Fix an arbitrary $0<\lambda\le\frac{2}{3}\lambda_0$.

For any $f\in\mathcal{F}(B)$ that is non-negative and harmonic on $B$, it is easy to check that
$$\check{f}=f_++2G^{\lambda_0^{-1}\lambda B}\left[\int_{B^c}f_-(y)J(\cdot,dy)\right]$$
is non-negative on $M$ and harmonic on $\lambda_0^{-1}\lambda B\subset\frac{2}{3}B$.

Applying $(\mathrm{wEH}_q^0)$ to $\check{f}$ (with $B$ replaced by $\lambda_0^{-1}\lambda B$), and then using $(\mathrm{E}_\le)$, we obtain
\begin{align*}
\left(\fint_{\lambda B}f^qd\mu\right)^{1/q}\le&\ \left(\fint_{\lambda B}\left(\check{f}\right)^qd\mu\right)^{1/q}\le C\mathop{\mathrm{einf}}\limits_{\lambda B}\check{f}\le C\left\{\mathop{\mathrm{einf}}\limits_{\lambda B}f_++\mathop{\mathrm{esup}}\limits_{\lambda_0^{-1}\lambda B}2G^{\lambda_0^{-1}\lambda B}\left[\int_{B^c}f_-(y)J(\cdot,dy)\right]\right\}\\
\le&\ C\left\{\mathop{\mathrm{einf}}\limits_{\lambda B}f+2\mathop{\mathrm{esup}}\limits_{z\in\lambda_0^{-1}\lambda B}\int_{B^c}f_-(y)J(z,dy)\cdot\mathop{\mathrm{esup}}\limits_{\frac{1}{2}B}G^{\lambda_0^{-1}\lambda B}1_{\lambda_0^{-1}\lambda B}\right\}\le C\mathop{\mathrm{einf}}\limits_{\lambda B}f+C'W(\lambda B)T_B^{(2/3)}(f_-).
\end{align*}
Thus $(\mathrm{wEH}_q)$ holds with $\lambda_0$ replaced by $\frac{2}{3}\lambda_0$. In the same way, we see $(\mathrm{EHI}^0)+(\mathrm{E}_\le)\Rightarrow(\mathrm{EHI})$.
\end{proof}

\section{Continuity of caloric functions with a weaker tail condition}\label{ctn}

\begin{lemma}\label{surv}
$(\mathrm{LLE})+(\mathrm{TJ})\Rightarrow(\mathrm{S}^+)$.
\end{lemma}

\begin{proof}
By (\ref{vLE}) and (\ref{s-c}) we see $(\mathrm{C})$ holds under $(\mathrm{LLE})$. Since $(\mathrm{TJ})$ holds, by \cite[Lemmas 7.9 and 9.7]{ghh+4} we obtain $(\mathrm{TP})$, that is, for every ball $B$ and all $t>0$,
$$P_t1_{B^c}\le\frac{Ct}{W(B)}.$$
The proof of $(\mathrm{TP})+(\mathrm{C})\Rightarrow(\mathrm{S}^+)$ completely repeats that of \cite[Lemma 10.3]{ghh+4}.
\end{proof}

Recall by \cite[Subsection 4.1]{ghh+3} that, with a sufficiently large $\beta>0$, there exist a constant $L_0\ge 1$ and a metric $d_\ast$ on $M$ (inducing the same topology as $d$ does) such that for all $x,y\in M$,
$$L_0^{-1}d_\ast(x,y)^\beta\le W(x,d(x,y))\le L_0d_\ast(x,y)^\beta.$$
Consequently, if $x,y\in\lambda B$ and $z\in(\lambda'B)^c$ with some $0<\lambda<\lambda'\le 1$, then with $C_L=4^{\beta_2}C_W^2L_0$,
\begin{align}
d_\ast(x,y)^\beta\le&L_0W(x,d(x,y))\le L_0C_WW(2B)\left(\frac{d(x,y)}{2R}\right)^{\beta_1}\le 4^{\beta_2-\beta_1}C_W^2L_0W\left(\frac{1}{2}B\right)\left(\frac{d(x,y)}{\frac{1}{2}R}\right)^{\beta_1};\label{*+}\\
d_\ast(x,z)^\beta\ge&L_0^{-1}W(x,d(x,z))\ge L_0^{-1}W(x,(\lambda'-\lambda)R)\ge L_0^{-1}C_W^{-1}(\lambda'-\lambda)^{\beta_2}W(B)\ge\frac{2^{\beta_1}(\lambda'-\lambda)^{\beta_2}}{C_W^2L_0}W\left(\frac{1}{2}B\right).\label{*-}
\end{align}

\begin{proposition}\label{sp-c}
Assume that conditions $(\mathrm{TJ})$, $(\mathrm{Gcap})$ and $(\mathrm{PI})$ hold. Then, for any ball $B=B(x_0,R)$ with $0<R<\overline{R}$, any $\delta_0W(\frac{1}{2}B)<T<W(x_0,\overline{R})$, $\varepsilon>0$ and $u:(0,T]\to\mathcal{F}'(B)$ that is caloric on $Q=(0,T]\times B$, there exist $\delta,\tau>0$ such that
\begin{equation}\label{uc}
|u(t,y)-u(t',z)|<\varepsilon\left(\|u\|_Q+\overline{T}_Q^{(2/3)}(u)\right)
\end{equation}
for all $t,t'\in\left[\delta_0W\left(\frac{1}{2}B\right),T\right]$ with $|t-t'|<\tau W\left(\frac{1}{2}B\right)$ and $\mu$-a.e.\ $y,z\in\frac{1}{2}B$ with $d(y,z)<\delta R$.
\end{proposition}

\begin{proof}
For simplicity we denote $w':=W(\frac{1}{2}B)$. If $\|u\|_Q=\infty$ or $\overline{T}_Q^{(2/3)}(u)=\infty$, then (\ref{uc}) holds trivially; while if $\|u\|_Q=0$, then $|u(t,y)-u(t',z)|\equiv 0$, so that (\ref{uc}) is also trivial. Thus it suffices to consider the case $0<\|u\|_Q+\overline{T}_Q^{(2/3)}(u)<\infty$.

For all $\delta_0w'\le t,t'\le T$ and $\mu$-a.e.\ $y,z\in\frac{1}{2}B$, say $u(t,y)\ge u(t',z)$, by applying (\ref{u-l}) to both $(\|u\|_Q\pm u)$ with $\lambda=\frac{2}{3}$ and using the basic fact $(\|u\|_Q-u)_-\le u_+$ and $(\|u\|_Q+u)_-\le u_-$, we obtain for all $0<s<t\wedge t'$ that
\begin{align}
u(t,y)-&\ u(t',z)=2\|u\|_Q-\left(\|u\|_Q-u(t,y)\right)-\left(\|u\|_Q+u(t',z)\right)\notag\\
\le&\ 2\|u\|_Q-\left(P_{t-s}^{\frac{2}{3}B}\left(\|u\|_Q-u(s,\cdot)\right)(y)-2\overline{T}_{(s,t]\times B}^{(2/3)}\left((\|u\|_Q-u)_-\right)\right)-\left(P_{t'-s}^{\frac{2}{3}B}\left(\|u\|_Q+u(s,\cdot)\right)(z)-2\overline{T}_{(s,t']\times B}^{(2/3)}\left((\|u\|_Q+u)_-\right)\right)\notag\\
\le&\ 2\|u\|_Q-\|u\|_QP_{t-s}^{\frac{2}{3}B}1_B(y)+P_{t-s}^{\frac{2}{3}B}u(s,\cdot)(y)+2\overline{T}_{(s,t]\times B}^{(2/3)}(u_+)-\|u\|_QP_{t-s}^{\frac{2}{3}B}1_B(z)-P_{t'-s}^{\frac{2}{3}B}u(s,\cdot)(z)+2\overline{T}_{(s,t']\times B}^{(2/3)}(u_-)\notag\\
\le&\ 2\|u\|_Q\mathop{\mathrm{esup}}\limits_{\frac{1}{2}B}\left(1-P_{t-s}^{\frac{2}{3}B}1_{\frac{2}{3}B}\right)+\left(P_{t-s}^{\frac{2}{3}B}u(s,\cdot)(y)-P_{t'-s}^{\frac{2}{3}B}u(s,\cdot)(z)\right)+2\overline{T}_{(s,t\vee t']\times B}^{(2/3)}(u)=:I_1+I_2+I_3.\label{Is}
\end{align}

We first estimate $I_1$. By Lemma \ref{surv}, $(\mathrm{S}^+)$ holds here. Thus (with covering technique)
\begin{equation}\label{I1}
I_1=2\|u\|_Q\mathop{\mathrm{esup}}\limits_{\frac{1}{2}B}\left(1-P_{t-s}^{\frac{2}{3}B}1_{\frac{2}{3}B}\right)\le 2\|u\|_Q\frac{C'_S(t-s)}{W(\frac{2}{3}B)}\le C_1\|u\|_Q\frac{t-s}{w'}.
\end{equation}

Now we estimate $I_2$. By \cite[Lemmas 6.9 and 6.11]{ghh+3}, there exist $C,\nu,\theta>0$ (independent of $B,t,t',s,y,z$) such that for all $f\in L^1\left(\frac{2}{3}B\right)$,
\begin{align}
\left|P_{t-s}^{\frac{2}{3}B}f(y)-P_{t-s}^{\frac{2}{3}B}f(z)\right|\le&\frac{C\left(\mathrm{diam}_{d_\ast}\left(\frac{2}{3}B\right)\right)^{\frac{\beta}{\nu}}}{\mu\left(\frac{2}{3}B\right)(t-s)^{\frac{1}{\nu}}}\exp\left\{\frac{t-s}{W\left(x_0,\bar{R}\right)}\right\}\left(\frac{d_\ast(y,z)}{(t-s)^{\frac{1}{\beta}}\wedge d_\ast\left(\frac{1}{2}B,\left(\frac{2}{3}B\right)^c\right)}\right)^\theta\|f\|_{L^1\left(\frac{2}{3}B\right)},\label{hker}\\
\left|P_{t-s}^{\frac{2}{3}B}f(z)-P_{t'-s}^{\frac{2}{3}B}f(z)\right|\le&\frac{CW\left(\frac{2}{3}B\right)^{\frac{1}{\nu}}\big|t-t'\big|}{\mu\left(\frac{2}{3}B\right)}\exp\left\{\frac{(t'\wedge t)-s}{W\left(x_0,\bar{R}\right)}\right\}\big((t'\wedge t)-s\big)^{-(1+\nu^{-1})}\|f\|_{L^1(\frac{2}{3}B)}.\label{hkns}
\end{align}
Noting $t,t'<W(x_0,\bar{R})$ and $y,z\in\frac{1}{2}B$, we have by (\ref{*+}) and (\ref{*-}) that
$$\frac{\left(d_\ast\left(\frac{1}{2}B,\left(\frac{2}{3}B\right)^c\right)\right)^\beta}{W\left(\frac{1}{2}B\right)}\ge\frac{2^{\beta_1}}{6^{\beta_2}C_W^2L_0},\quad\frac{\left(\mathrm{diam}_{d_\ast}\left(\frac{2}{3}B\right)\right)^\beta}{W\left(\frac{1}{2}B\right)}\le 4^{\beta_2-\beta_1}C_W^2L_0\left(\frac{8}{3}\right)^{\beta_1},\quad\frac{d_\ast(y,z)^\beta}{W\left(\frac{1}{2}B\right)}\le 4^{\beta_2-\beta_1}C_W^2L_0\left(\frac{d(y,z)}{\frac{1}{2}R}\right)^{\beta_1}.$$
Combining these observations with (\ref{hker}) and (\ref{hkns}), we obtain by taking $f=u(s,\cdot)$ that
\begin{align}
I_2=&\ P_{t-s}^{\frac{2}{3}B}u(s,\cdot)(y)-P_{t'-s}^{\frac{2}{3}B}u(s,\cdot)(z)\le\left|P_{t-s}^{\frac{2}{3}B}u(s,\cdot)(y)-P_{t-s}^{\frac{2}{3}B}u(s,\cdot)(z)\right|+\left|P_{t-s}^{\frac{2}{3}B}u(s,\cdot)(z)-P_{t'-s}^{\frac{2}{3}B}u(s,\cdot)(z)\right|\notag\\
\le&\ \frac{C_2}{\mu\left(\frac{2}{3}B\right)}\left\{\left(\frac{W\left(\frac{1}{2}B\right)}{t-s}\right)^{\frac{1}{\nu}+\frac{\theta}{\beta}}\left(\frac{d(y,z)}{R}\right)^{\frac{\beta_1\theta}{\beta}}+\frac{W\left(\frac{1}{2}B\right)^{\frac{1}{\nu}}\big|t-t'\big|}{((t'\wedge t)-s)^{1+\frac{1}{\nu}}}\right\}\int_{\frac{2}{3}B}u(s,\cdot)d\mu\notag\\
\le&\ C_2\|u\|_Q\left\{\left(\frac{w'}{t-s}\right)^{\frac{1}{\nu}+\frac{\theta}{\beta}}\left(\frac{d(y,z)}{R}\right)^{\frac{\beta_1\theta}{\beta}}+\frac{(w')^{\frac{1}{\nu}}\big|t-t'\big|}{((t'\wedge t)-s)^{1+\frac{1}{\nu}}}\right\}.\label{I2}
\end{align}

Now we turn to $I_3$. Since $\overline{T}_Q^{(2/3)}(u)=\int_0^TT_B^{(2/3)}(u;t)dt<\infty,$ by the absolute continuity of integration, for each $\varepsilon'>0$, there exists $\delta_\ast=\delta_\ast(\varepsilon';u,B)>0$ such that when $(t\vee t')-\delta_\ast w'<s<t\wedge t'$,
\begin{equation}\label{I3}
I_3=2\int_s^{t\vee t'}T_B^{(2/3)}(u;s')ds'<\varepsilon'\overline{T}_Q^{(2/3)}(u).
\end{equation}

Given any $\varepsilon>0$, take
$$\tau_\ast:=\frac{\delta_0}{4}\wedge\frac{\varepsilon}{4C_1}\wedge\frac{1}{2}\delta_\ast\left(\varepsilon;u,B\right),\quad\tau^\ast=\frac{\varepsilon}{4C_2}\tau_\ast^{1+\frac{1}{\nu}}\quad\mbox{and}\quad\delta=\left(\frac{\varepsilon}{4C_2}\tau_\ast^{\frac{1}{\nu}+\frac{\theta}{\beta}}\right)^{\frac{\beta}{\beta_1\theta}}.$$
Then for any $|t-t'|<(\tau_\ast\wedge\tau^\ast)w'$ and $d(y,z)<\delta R$, with $s:=(t\wedge t')-\tau_\ast w'$ we see
$$(t\vee t')-s=|t-t'|+(t\wedge t')-s<2\tau_\ast w'.$$
In particular, (\ref{I3}) holds with $\varepsilon'=\frac{\varepsilon}{4}$, and $I_1<\frac{\varepsilon}{2}\|u\|_Q$ by (\ref{I1}). On the other hand, by (\ref{I2}),
$$I_2<C_2\|u\|_Q\left\{\tau_\ast^{-\frac{1}{\nu}+\frac{\theta}{\beta}}\delta^{\frac{\beta_1\theta}{\beta}}+\tau^\ast\tau_\ast^{-1-\frac{1}{\nu}}\right\}=\left(\frac{\varepsilon}{4}+\frac{\varepsilon}{4}\right)\|u\|_Q=\frac{\varepsilon}{2}\|u\|_Q.$$
Thus from (\ref{Is}) we obtain
$$|u(t,y)-u(t,z)|\le I_1+I_2+I_3<\frac{\varepsilon}{2}\|u\|_Q+\frac{\varepsilon}{2}\|u\|_Q+\varepsilon\overline{T}_Q^{(2/3)}(u)\le\varepsilon\left(\|u\|_Q+\overline{T}_Q^{(2/3)}(u)\right),$$
which completes the proof.
\end{proof}

This proposition shows that a function $u$, which is bounded and caloric on $Q$ with a finite $L^1$-tail, is uniformly continuous on any strict subcylinder of $Q$. Be careful that the uniformity depends irregularly on $B$ and $u$. Now we show the irregularity is essential:

\begin{lemma}
Assume that $(\mathrm{S}^-)$ holds. If $B$ is a metric ball and $J(x,U)$ is bounded from both sides for all $x\in B$ with some $U\Subset B^c$, then there exist caloric functions $u_n$ on $(0,\infty)\times B$ such that $\|u_n\|_{(0,\infty)\times B}+\overline{T}_{(0,\infty)\times(\overline{U})^c}^{B}(u_n)$ is uniformly bounded, but there exists a sequence $\{t_n\}_{n=1}^\infty$ that is convergent to some $t_0\in(0,\infty)$ such that for all sufficiently large $n\ge 1$,
\begin{equation}\label{irr}
\mathop{\mathrm{einf}}\limits_{x\in\frac{1}{4}B}\big|u_n(t_0,x)-u_n(t_n,x)\big|\ge\frac{1}{4}.
\end{equation}
\end{lemma}

Note that for a regular stable-like jump form with
$$J(x,y)\asymp\frac{1}{V(x,d(x,y))W(x,d(x,y))},$$
given any two different points $x_0,y_0$, this lemma can always be applied to $B=B(x_0,d(x_0,y_0)/3)$ and $U=B(y_0,d(x_0,y_0)/3)$.

\begin{proof}
By rescaling we assume $W(B)=1$ and $\frac{1}{2}\le J(x,U)\le 2$ for all $x\in B$. Let
$$f_n(t,x)=n\cdot 1_{1-\frac{1}{n}<t\le 1,x\in U}.$$
Clearly
$$T_{(0,\infty)\times(\overline{U})^c}^{B;p}(f_n)=\left(\int_{1-\frac{1}{n}}^1\left(T_{(\overline{U})^c}^B(f_n;s)\right)^pds\right)^{1/p}\le\left(\int_{1-\frac{1}{n}}^1\left(2n\right)^pds\right)^{1/p}=2n^{1-1/p}<\infty$$
for all $p\in[1,\infty)$. Hence by Lemma \ref{ext-F}, the caloric extension $u_n$ of $f_n$ into $(0,\infty)\times B$ is in $\mathcal{F}'(B)$ and
$$\|u_n\|_{(0,\infty)\times B}++\overline{T}_{(0,\infty)\times(\overline{U})^c}^{B}(u_n)\le 3\overline{T}_{(0,\infty)\times(\overline{U})^c}^B(f_n)\le 6.$$
However, by $(\mathrm{S}^-)$ there exists $\delta>0$ such that $P_t^B1_B\ge\frac{1}{2}$ on $\frac{1}{4}B$ for all $0<t<\delta$. Therefore, for all $n\ge\delta^{-1}$ and $\mu$-a.e.\ $x\in\frac{1}{4}B$,
$$u_n(1,x)=\int_{1-\frac{1}{n}}^1\int_B\int_UnJ(z,dy)p_{t-s}^B(x,dz)ds\ge\frac{n}{2}\int_{1-\frac{1}{n}}^1P_{t-s}^B1_B(x)ds\ge\frac{1}{4}.$$
On the other hand, trivially $u_n(1-\frac{1}{n},x)=0$. Thus (\ref{irr}) is proved with $t_n=1-\frac{1}{n}$.
\end{proof}

Meanwhile, if $T_Q^{(2/3;p)}(u)<\infty$ with some $p>1$, then the result is greatly improved:

\begin{proposition}\label{p-phr}
Under the same conditions as in Proposition \ref{sp-c}, for any $p>1$ and $\delta_0>0$, there exist $\delta,\tau>0$ such that for every ball $B$, all $u$ that is caloric on $Q=(0,T]\times B$ with an arbitrary $T>\delta_0W\left(\frac{1}{2}B\right)$, all $t,t'\in\left[\delta_0W\left(\frac{1}{2}B\right),T\right]$ and $\mu$-a.e.\ $y,z\in\frac{1}{2}B$ such that $|t-t'|<\tau W\left(\frac{1}{2}B\right)$ and $d(y,z)<\delta R$,
$$|u(t,y)-u(t',z)|<\varepsilon\left(\|u\|_Q+W(B)^{1-\frac{1}{p}}T_Q^{(\frac{2}{3};p)}(u)\right).$$
Further, $(\mathrm{PHR})$ holds, that is, $(\mathrm{TJ})+(\mathrm{Gcap})+(\mathrm{PI})\Rightarrow(\mathrm{PHR})$.
\end{proposition}

\begin{proof}
It suffices to prove $(\mathrm{PHR})$, and the first follows by applying $(\mathrm{PHR})$ with $t_0=t\vee t'$ to $B$ (if $t_0\ge W(B)$) or $B(z,\frac{1}{2}W^{-1}(z,t_0))$ (if else).

Let $I_1,I_2,I_3$ and $w'$ be the same as in (\ref{Is}). Without loss of generality, we assume that
$$\mathcal{A}_u:=\|u\|_Q+W(B)^{1-\frac{1}{p}}T_Q^{(\frac{2}{3};p)}(u)\in(0,\infty)$$
and $\theta<\beta$ in (\ref{I2}). Fix $\delta_0=2^{-1/\beta_2}C_W^{-1}$, and denote
\begin{equation}\label{XY}
X=\frac{t-t'}{w'},\quad Y=\frac{d(y,z)}{R},\quad\xi=\delta_0\left(X^{-1-\frac{1}{\nu}+\frac{1}{p}}\left(Y^{\frac{\beta_1\theta}{\beta}}X^{-\frac{\theta}{\beta}}+1\right)\right)^{\frac{1}{\frac{1}{\nu}+\frac{\theta}{\beta}+1-\frac{1}{p}}}
\end{equation}
for all $t,t'\in[\delta_0W(B),W(B)]$ (say $t\ge t'$) and $\mu$-a.e.\ $y,z\in\frac{1}{2}B$.

By choosing
$$\delta=C_W^{-\frac{1}{\beta_1}}2^{-1-\frac{\beta}{\theta\beta_1}}\quad\mbox{and}\quad\gamma=\frac{\beta_1\frac{\theta}{\beta}}{\frac{1}{\nu}+\frac{\theta}{\beta}+1-\frac{1}{p}},$$
we see $X<2^{-\beta/\theta}$ for all $t,t'\in(W(B)-W(\delta B),W(B)]$ and $Y<2^{-\beta/(\theta\beta_1)}$ for all $y,z\in\delta B$. Thus clearly, $\xi\ge\delta_0$ and
$$s:=t'-\xi Xw'=t'-\delta_0w'\left(Y^{\frac{\beta_1\theta}{\beta}}+X^{\frac{\theta}{\beta}}\right)^{\frac{1}{\frac{1}{\nu}+\frac{\theta}{\beta}+1-\frac{1}{p}}}>\delta_0w'-\delta_0w'\left(\frac{1}{2}+\frac{1}{2}\right)^{\frac{1}{\frac{1}{\nu}+\frac{\theta}{\beta}+1-\frac{1}{p}}}=0.$$
Therefore, by (\ref{I1}), (\ref{I2}) and H\"older inequality for $I_3$ we have
\begin{align*}
\big|u(t,y)&-u(t',z)\big|\le I_1+I_2+I_3\\
\le&\ \frac{C_1\|u\|_Q}{w'}(t-s)+C_2\|u\|_Q\left\{\left(\frac{w'}{t'-s}\right)^{\frac{1}{\nu}+\frac{\theta}{\beta}}\left(\frac{d(y,z)}{R}\right)^{\frac{\beta_1\theta}{\beta}}+\frac{w'^{\frac{1}{\nu}}\big(t-t'\big)}{(t'-s)^{1+\frac{1}{\nu}}}\right\}+2(t-s)^{1-\frac{1}{p}}T_Q^{(2/3;p)}(u)\\
\le&\ C\left(\frac{t-s}{w'}+\left(\frac{w'}{t'-s}\right)^{\frac{1}{\nu}+\frac{\theta}{\beta}}\left(\frac{d(y,z)}{R}\right)^{\frac{\beta_1\theta}{\beta}}+\frac{w'^{\frac{1}{\nu}}\big(t-t'\big)}{(t'-s)^{1+\frac{1}{\nu}}}+\left(\frac{t-s}{w'}\right)^{1-\frac{1}{p}}\right)\mathcal{A}_u\\
=&\ C\left((1+\xi)X+\frac{Y^{\frac{\beta_1\theta}{\beta}}}{(\xi X)^{\frac{1}{\nu}+\frac{\theta}{\beta}}}+\frac{1}{\xi^{1+\frac{1}{\nu}}X^{\frac{1}{\nu}}}+\left((1+\xi)X\right)^{1-\frac{1}{p}}\right)\mathcal{A}_u\le C'(\delta_0)\left((\xi X)^{1-\frac{1}{p}}+\frac{Y^{\frac{\beta_1\theta}{\beta}}X^{-\frac{1}{\nu}-\frac{\theta}{\beta}}+X^{-\frac{1}{\nu}}}{\xi^{\frac{1}{\nu}+\frac{\theta}{\beta}}}\right)\mathcal{A}_u\\
=&\ 2C'(\delta_0)\delta_0\left(Y^{\frac{\beta_1\theta}{\beta}}+X^{\frac{\theta}{\beta}}\right)^{\frac{1}{\frac{1}{\nu}+\frac{\theta}{\beta}+1-\frac{1}{p}}}\mathcal{A}_u\le C\left(\delta_0,\frac{\theta}{\beta}\right)\left(Y+X^{\frac{1}{\beta_1}}\right)^{\frac{\beta_1\frac{\theta}{\beta}}{\frac{1}{\nu}+\frac{\theta}{\beta}+1-\frac{1}{p}}}\mathcal{A}_u\le C'\left(\frac{d(y,z)+W^{-1}(z,|t-t'|)}{R}\right)^\gamma\mathcal{A}_u,
\end{align*}
which is exactly the condition $(\mathrm{PHR})$.
\end{proof}

\begin{appendices}
\section{Joint measurabilty}

Here only $(\mathrm{VD})$ is assumed on $(M,d,\mu)$. Given $0<\varepsilon\le\varepsilon'<\infty$ and an open set $\Omega$ in $M$, we call $\{y_\iota\}_{\iota\in I}$ as an $(\varepsilon',\varepsilon)$-net of $\Omega$, if each $y_\iota\in\Omega$; for every two different indices $\iota,\iota'\in I$, $d(y_\iota,y_{\iota'})\ge\varepsilon$; and for all $x\in\Omega$, there exists $\iota\in I$ such that $d(x,y_\iota)<\varepsilon'$.

\begin{lemma}\label{nr-net}
There exists $C_\ast>0$ such that for every bounded open set $\Omega$ in $M$ and every $\varepsilon>0$, there exists a $(5\varepsilon,\varepsilon)$-net of $\Omega$ containing no more than $C_\ast\left\lceil\left(\frac{\mathrm{diam}(\Omega)}{\varepsilon}\right)^\alpha\right\rceil+1$ elements.
\end{lemma}

\begin{proof}
The result trivially holds if $\Omega$ is a singleton (where $\Omega$ itself is a $(5\varepsilon,\varepsilon)$-net of $\Omega$). Hence we only consider the case $\mathrm{diam}(\Omega)>0$.

Fix arbitrary $y_0\in\Omega$ and set $R=2\mathrm{diam}(\Omega)$. Then clearly $\Omega\in B(y_0,R)$. Note that $\{B(x,\varepsilon)\}_{x\in\Omega}$ covers $\Omega$. By \cite[Theorem 1.2]{hein}, there are points $x_i$ ($i\in I$, where $I$ is an arbitrary index set) such that $B(x_i,\varepsilon)\cap B(x_j,\varepsilon)=\emptyset$ for any $i\ne j$, and for any $y\in\Omega$ there exists $i$ such that $y\in B(x_i,5\varepsilon)$. It follows that $d(x_i,x_j)\ge\varepsilon$. Suppose there are $N$ different $x_i$'s, say $x_1,\dotsc,x_N$. Then
$$\sum\limits_{i=1}^NV(x_i,\varepsilon)=\mu\left(\bigcup\limits_{i=1}^NB(x_i,\varepsilon)\right)\le V(y_0,R+\varepsilon),$$
which yields with the help of $(\mathrm{VD})$ that
$$N\le\max_{1\le i\le N}\frac{V(y_0,R+\varepsilon)}{V(x_i,\varepsilon)}\le\max_{1\le i\le N}C_{\mathrm{VD}}\left(\frac{d(y_0,x_i)+R+\varepsilon}{\varepsilon}\right)^\alpha\le 2C_{\mathrm{VD}}\cdot 4^\alpha\left\lceil\left(\frac{\mathrm{diam}(\Omega)}{\varepsilon}\right)^\alpha\right\rceil.$$
The proof is then completed with $C_\ast=2^{1+2\alpha}C_{\mathrm{VD}}$.
\end{proof}

\begin{lemma}\label{meas}
Let $\Omega$ be an arbitrary open set in $M$. Then for any $u:(t_1,t_2]\to L^1(\Omega)$, if $t\mapsto\int_\Omega u(t,\cdot)\varphi d\mu$ is continuous on $(t_1,t_2]$ for every $\varphi\in C_0(\Omega)$, then $u$ admits a version that is $\hat{\mu}$-measurable on $Q:=(t_1,t_2]\times\Omega$.
\end{lemma}

In particular, any function $u$ that is (sub/super)caloric in the sense of \cite{ghh+2} is globally $\hat{\mu}$-measurable.

\begin{proof}
It suffices to assume that $\Omega$ is bounded (otherwise we focus on $\Omega\cap B(x_0,n)$ for every $n\ge 1$ and then take the limit). Thus by Lemma \ref{nr-net}, for each $n\ge 0$, there exists a $(5^{-n},5^{-n-1})$-net $\{y_{i,n}\}_{1\le i\le M_n}$ of $\Omega$. Let
$$V_{i,n}:=B(y_{i,n},5^{-n})\cap\Omega\quad\mbox{and}\quad U_{i,n}:=B\left(y_{i,n},5^{-n-1}\right)\cap\left\{x:d(x,\Omega^c)>\frac{5^{-n-1}}{2}\wedge\frac{d(y_{i,n},\Omega^c)}{2}\right\}.$$
Fix $\varphi_{i,n}\in\mathrm{cutoff}(U_{i,n},V_{i,n})$, which exists and satisfies $\|\varphi_{i,n}\|_1>0$ since
$$V_{i,n}\Supset U_{i,n}\supset B\left(y_{i,n},5^{-n-1}\wedge 2^{-1}d(y_{i,n},\Omega^c)\right).$$
For each $x\in\Omega$ and $n\ge 0$, let $i_1^{(n)}(x),\cdots,i_{k(n,x)}^{(n)}(x)$ be all the integers $i\in\{1,\dotsc,M_n\}$ such that $d(x,y_{i,n})<5^{-n}$. Define
$$u_n(t,x):=\sum\limits_{k=1}^{k(n,x)}\int u(t,z)\varphi_{i_k^{(n)},n}(z)d\mu(z)\left(\sum\limits_{k'=1}^{k(n,x)}\int\varphi_{i_{k'}^{(n)},n}(z')d\mu(z')\right)^{-1}.$$

By the assumptions on $u$, we see easily that $u_n(t,\cdot)$ is $\mu$-measurable for any $t\in(t_1,t_2]$, while $u_n(\cdot,x)$ is continuous in $t$ for any $x\in\Omega$. It follows that $u_n$ is $\hat{\mu}$-measurable on $Q$. Further, for each $t\in(t_1,t_2]$, parallel to \cite[Theorem 1.8]{hein} (Lebesgue's differentiation lemma), there exists $N_t\subset\Omega$ such that $\mu(N_t)=0$, and for each $t_1<t\le t_2$ and $x\in\Omega$, $\lim\limits_{n\to\infty}u_n(t,x)$ exists if and only if $x\notin N_t$. Let
$$\tilde{u}(t,x):=\begin{cases}\lim\limits_{n\to\infty}u_n(t,x)&\mbox{for all}\quad x\in\Omega\setminus N_t;\\
0&\mbox{for all}\quad x\in N_t.\end{cases}$$
By \cite[Theorem 1.8]{hein}, $\tilde{u}(t,\cdot)=u(t,\cdot)$ for $\mu$-a.e.\ in $\Omega$ with any $t\in(t_1,t_2]$. Since
$$\mathcal{N}_Q:=\{(t,x)\in Q:x\in N_t\}=\left\{(t,x):\varlimsup\limits_{n\to\infty}u_n(t,x)\ne\varliminf\limits_{n\to\infty}u_n(t,x)\right\}$$
is $\hat{\mu}$-measurable, then
$$\hat{\mu}\left(\mathcal{N}_Q\right)=\int_{t_1}^{t_2}\mu(N_t)dt=0$$
by Tonelli theorem. Thus $\tilde{u}$ is $\hat{\mu}$-measurable on $Q$, which is then a desired version of $u$.
\end{proof}

\section{The stochastic counterpart}\label{prob}

Let $M_\partial:=M\cup\{\partial\}$ be the one-point compactification of $M$, and $\mathbb{P}^x$ be a Borel probability on $M_\partial$ for each $x\in M_\partial$. A strong Markov process $\left(X,\{\mathbb{P}^x\}_{x\in M_\partial}\right)$ is called a \emph{Hunt process} on $M$, if $\mathbb{P}^x[X_0=x]=1$ for all $x\in M$; $\mathbb{P}^\partial[X_t=\partial]=1$ for all $t\ge 0$; $t\mapsto X_t(\omega)$ is right continuous and admits left limit for every $\omega$; and $X$ is quasi-left continuous, that is, $\mathbb{P}^\nu[X_{\sigma_n}\to X_\sigma\ |\ \sigma<\infty]=1$ for any Borel probability $\nu$ on $M_\partial$ and any stopping times $\sigma_n\uparrow\sigma$. Here $\mathbb{P}^\nu(\cdot):=\int_{M_\partial}\mathbb{P}^x(\cdot)d\nu(x)$. For any Borel set $A\subset M_\partial$, define the \emph{hitting} and \emph{exit times} of $A$ respectively by
$$\sigma_A=\inf\big\{t\ge 0:X_t\in A\big\},\quad\tau_A=\inf\big\{t\ge 0:X_t\notin A\big\}=\sigma_{M_\partial\setminus A}.$$

Given a heat semigroup $P_t$ on $L^2(M,\mu)$, there exists a unique Hunt process $X$ whose transition law is given by $\mathbb{P}^x[X_t\in A]=P_t1_A(x)$ for every $\mu$-measurable set $A\subset M$, all $t>0$ and $x\in M\setminus\mathcal{N}$, where $\mathcal{N}$ is a \emph{properly exceptional} set, that is, $\mathbb{P}^\nu[\sigma_\mathcal{N}<\infty]=0$ for every Borel probability $\nu$.

For any measurable function $f$ on $M$, we set $f(\partial)=0$. This way,
$$P_t^\Omega f(x)=\mathbb{E}^x\left[f\left(X_t^\Omega\right)\right],\quad\mbox{where}\quad X_t^\Omega(\omega):=\begin{cases}X_t(\omega),&\mbox{if }t<\tau_\Omega(\omega)\\
\partial,&\mbox{if }t\ge\tau_\Omega(\omega)\end{cases}$$
for any $f\in L^1(\Omega)$ and any open set $\Omega$ in $M$ by \cite[Formulae (3.3.1) and (3.3.2)]{cf}. In particular,
$$P_t^\Omega 1_\Omega(x)=\mathbb{E}^x\left[1_\Omega\left(X_t^\Omega\right)\right]=\mathbb{P}^x\left[X_t^\Omega\in\Omega\right]=\mathbb{P}^x\left[\tau_\Omega>t\right]\quad\mbox{and}\quad G^\Omega1_\Omega(x)=\int_0^\infty P_t^\Omega 1_\Omega(x)dt=\mathbb{E}^x\left[\tau_\Omega\right],$$
which is the reason why conditions $(\mathrm{S}^\pm)$, $(\mathrm{E})$ are called ``survival'' and ``exit time'' estimates.

Further, take $\hat{X}_s=(V_0-s,X_s)$, where $\mathbb{P}^{(t,x)}[V_0=t]=1$ for all $t\in\mathbb{R}$ and $x\in M$. Denote
$$\hat{\sigma}_D=\inf\left\{s\ge 0:\hat{X}_s\in D\right\}\quad\mbox{and}\quad\hat{\tau}_D=\inf\left\{s\ge 0:\hat{X}_s\notin D\right\}$$
for every Borel set $D$ in $\mathbb{R}\times M$. A function $u:(t_1,t_2]\times M\to\mathbb{R}$ is called $X$-\emph{caloric} on $Q=(t_1,t_2]\times\Omega$, if $u$ is nearly Borel measurable (that is, for any $a\in\mathbb{R}$, there exist Borel sets $A_1,A_2\subset Q_\infty:=(t_1,t_2]\times M$ such that $A_1\subset\{(t,x)\in Q_\infty:u(t,x)\le a\}\subset A_2$ and $\mathbb{P}^\nu\left[\hat{\sigma}_{A_2\setminus A_1}<\infty\right]=0$ for all Borel measure $\nu$ on $Q_\infty$), and there exists a properly exceptional set $\mathcal{N}_u\subset M$ such that for any open subset $D\Subset Q$, all $t_1<t\le t_2$ and $x\in M_\partial\setminus\mathcal{N}_u$,
$$u(t,x)=\mathbb{E}^{(t,x)}\left[u\left(\hat{X}_{\hat{\tau}_D}\right)\right].$$

Recall by \cite[Formula (A.3.33)]{cf} that $(2J(x,dy),t)$ is a \emph{L\'evy system} of the Hunt process $X$, that is, for any stopping time $\tau$, any Borel functions $g:\mathbb{R}_+\to\mathbb{R}_+$ and $\chi:M_\partial\times M_\partial\to\mathbb{R}$ such that $\chi(x,x)=0$ for all $x\in M_\partial$,
\begin{equation}\label{levy}
\mathbb{P}^x\left[\sum\limits_{s\le\tau}g(s)\chi(X_{s-},X_s)\right]=\mathbb{E}^x\left[\int_0^\tau g(s)\int_{M_\partial}\chi(X_s,z)2J(X_s,dz)ds\right].
\end{equation}
Here we emphasize that, according to \cite[Formulae (4.3.14) and (4.3.16)]{cf}, the term $J(x,dy)$ in Chen's works \cite{cf,ckw-p} is equal to our $2J(x,dy)$ (though there is misleading information in \cite[Formula (1.1)]{ckw-p}). That is why the L\'evy system is expressed differently.

From (\ref{levy}) and the strong Markov property of $X$, we could easily check that:

\begin{lemma}\label{exX}
The caloric extension $f_{t_1,\lambda B}$ is always $X$-caloric on $(t_1,t_2]\times\lambda B$.\qed
\end{lemma}
In comparison, $f_{t_1,\lambda B}$ is caloric when $T_{(t_1,t_2]\times B}^{(\lambda;2)}(f)<\infty$, but it may not be in $\mathcal{F}'(\Omega)$ otherwise. Hence we cannot expect a general $X$-caloric function to be (analytically) caloric.

Meanwhile, by \cite[Remark 1.12]{ckw-p}, for any two open sets $U\subset\Omega$ in $M$, $P_t^\Omega f$ is always $X$-caloric on $(0,\infty)\times U$. Following the same idea, we have:

\begin{proposition}\label{AIS}
Assume that $(\mathcal{E},\mathcal{F})$ is a regular Dirichlet form. Given arbitrary $\lambda\in(0,1)$ and cylinder $Q=(t_1,t_2]\times B$, where $B$ is a metric ball in $M$, then any function $u$ that is continuous and caloric on $Q$ with $\overline{T}_Q^{(\lambda)}(u)<\infty$ is $X$-caloric on $(t_1,t_2]\times\lambda B$.
\end{proposition}

\begin{proof}
Without loss of generality, we assume $s_0=t_2$ so that $\frac{d}{dt}(u(t,\cdot),\varphi)$ exists for any $\varphi\in\mathcal{F}(B)$ and $t_1<t\le t_2$.

Step 1. Consider the case $T_Q^{(\lambda)}(u)<\infty$. For each $0<\lambda'<\lambda$, take
$$\phi_{\lambda'}\in\mathrm{cutoff}\left(\frac{\lambda+\lambda'}{2}B,\lambda B\right)\quad\mbox{and}\quad u_{t;\lambda'}(x):=u(t,x)-\left(u(1-\phi_{\lambda'})\right)_{t_1,\lambda'B}(t,x).$$
Then $(t,x)\mapsto u_{t;\lambda'}(x)$ is caloric on $(t_1,t_2]\times\lambda'B$ and is equal to $u(t,\cdot)\phi_{\lambda'}$ on $(\lambda'B)^c$. In particular, $u_{t;\lambda'}\in\mathcal{F}(\lambda B)$. Therefore, for any $\varphi\in\mathcal{F}(\lambda'B)$ and $t_1<t\le t_2$, we see
\begin{align*}
\lim\limits_{s\downarrow 0}\frac{1}{s}\mathbb{E}^{\varphi\cdot\mu}&\left[u_{t;\lambda'}\left(X_s^{\lambda B}\right)-u_{t;\lambda'}\left(X_0^{\lambda B}\right)\right]=\lim\limits_{s\downarrow 0}\frac{1}{s}\int_{\lambda B}\left(\mathbb{E}^x\left[u_{t;\lambda'}\left(X_s^{\lambda B}\right)\right]-u_{t;\lambda'}(x)\right)\varphi(x)d\mu(x)\\
=&\ \lim\limits_{s\downarrow 0}\frac{1}{s}\left(P_s^{\lambda B}u_{t;\lambda'}-u_{t;\lambda'},\varphi\right)=\mathcal{E}(u_{t;\lambda'},\varphi)=-\frac{\mathrm{d}}{\mathrm{d}t}(u_{t;\lambda'},\varphi)=\lim\limits_{s\downarrow 0}\frac{1}{s}\mathbb{E}^{\varphi\cdot\mu}\left[u_{t-s;\lambda'}\left(X_0^{\lambda B}\right)-u_{t;\lambda'}\left(X_0^{\lambda B}\right)\right].
\end{align*}
In particular, since $\cup_{0<\lambda'<\lambda}\mathcal{F}(\lambda'B)$ is dense in $L^2(\lambda B)$, then in the weak $L^2(\lambda B)$ sense,
\begin{equation}\label{cal2}
\lim\limits_{\lambda'\uparrow\lambda,s\downarrow 0}\frac{1}{s}\mathbb{E}^\bullet\left[u_{t;\lambda'}\left(X_s^{\lambda B}\right)-u_{t-s;\lambda'}\left(X_0^{\lambda B}\right)\right]=0.
\end{equation}

For any $t_1<t_0\le t_2$, $0\le s\le t<t_1-t_0$ and $\mu$-a.e.\ $x\in\lambda' B$, by Lemma \ref{exX} we obtain
\begin{align*}
\mathbb{E}^{t_0,x}&\left[u\left(\hat{X}_{t\wedge\tau_{\lambda B}}\right)\big|\mathfrak{F}_{s\wedge\tau_{\lambda B}}\right]=\mathbb{E}^x\left[u\left(t_0-t,X_t^{\lambda B}\right)\big|\mathfrak{F}_{s\wedge\tau_{\lambda B}}\right]=\mathbb{E}^x\left[\left(\left(u(1-\phi_{\lambda'})\right)_{t_1,\lambda B}+u_{t_0-t;\lambda'}\right)\left(X_t^{\lambda B}\right)\big|\mathfrak{F}_{s\wedge\tau_{\lambda B}}\right]\\
&=\left(u(1-\phi_{\lambda'})\right)_{t_1,\lambda B}\left(t_0-s,X_s^{\lambda B}\right)+\mathbb{E}^{X_s^{\lambda B}}\left[u_{t_0-t;\lambda'}\left(X_{t-s}^{\lambda B}\right)\right]\\
&=\left(u(1-\phi_{\lambda'})\right)_{t_1,\lambda B}\left(t_0-s,X_s^{\lambda B}\right)+\mathbb{E}^{X_s^{\lambda B}}\left[u_{t_0-s;\lambda'}\left(X_0^{\lambda B}\right)\right]+\int_s^t\frac{\mathrm{d}}{\mathrm{d}t'}\mathbb{E}^{X_s^{\lambda B}}\left[u_{t_0-t';\lambda'}\left(X_{t'}^{\lambda B}\right)\right]dt'\\
&=\left(u(1-\phi_{\lambda'})\right)_{t_1,\lambda B}\left(t_0-s,X_s^{\lambda B}\right)+u_{t_0-s;\lambda'}\left(X_s^{\lambda B}\right)+\int_s^t\lim\limits_{s'\downarrow 0}\frac{1}{s'}\mathbb{E}^{X_s^{\lambda B}}\left[u_{t_0-t'-s';\lambda'}\left(X_{t'+s'}^{\lambda B}\right)-u_{t_0-t';\lambda'}\left(X_{t'}^{\lambda B}\right)\right]dt'\\
&=u\left(t_0-s,X_s^{\lambda B}\right)+\int_s^tP_{t'}^{\lambda B}\left\{\lim\limits_{s'\downarrow 0}\frac{1}{s'}\mathbb{E}^\bullet\left[u_{t_0-t'-s';\lambda'}\left(X_{s'}^{\lambda B}\right)-u_{t_0-t';\lambda'}\left(X_0^{\lambda B}\right)\right]\right\}\left(X_s^{\lambda B}\right)dt'.
\end{align*}
Taking $\lambda'\uparrow\lambda$ and combining (\ref{cal2}), we see for $\mu$-a.e.\ $x\in\lambda B$ that
$$\mathbb{E}^{t_0,x}\left[u\left(\hat{X}_{t\wedge\tau_{\lambda B}}\right)\big|\mathfrak{F}_{s\wedge\tau_{\lambda B}}\right]=u\left(t_0-s,X_s^{\lambda B}\right)=u\left(\hat{X}_{s\wedge\tau_{\lambda B}}\right),$$
which holds for all $x\in\lambda B\setminus\mathcal{N}$ by the continuity of $u$ and quasi-continuity of $X$.

By the optimal sampling property for the martingale $u\left(\hat{X}_{\bullet\wedge\tau_{\lambda B}}\right)$, it follows that
$$u(t_0,x)=\mathbb{E}^{(t_0,x)}\left[u\left(\hat{X}_{\hat{\tau}_D\wedge\tau_{\lambda B}}\right)\right]=\mathbb{E}^{(t_0,x)}\left[u\left(\hat{X}_{\hat{\tau}_D}\right)\right]$$
for any $D\Subset(t_1,t_2]\times\lambda B$. That is, $u$ is $X$-caloric on $(t_1,t_2]\times\lambda B$, with $\mathcal{N}_u=\mathcal{N}$.

Step 2. For general $u$, take $u_\varepsilon$ for any $0<\varepsilon<t_2-t_1$ in the same way as in (\ref{uep}), which is then caloric, continuous on $Q_\varepsilon:=(t_1+\varepsilon,t_2]\times B$ and $T_{Q_\varepsilon}^{(\lambda)}(u_\varepsilon)<\infty$. Thus $u_\varepsilon$ is $X$-caloric on $(t_1+\varepsilon,t_2]\times\lambda B$ by Step 1. The desired result follows by taking $\varepsilon\downarrow 0$.
\end{proof}

Note that the continuity of $u$ is only used to ensure that $\mathcal{N}_u$ is independent of $t$.
\end{appendices}

\subsection*{Acknowledgment}
The author would like to thank Professor Jiaxin Hu for introducing the topic, and also for discussions on the growth lemmas. The author is grateful to Professors Alexander Grigor'yan, Moritz Kassmann and other colleagues for suggestions on this manuscript.

\end{document}